\documentclass[notitlepage]{amsart}

\newtheorem{pro}{Proposition}[section]
\newtheorem{teo}[pro]{Theorem}
\newtheorem{defi}[pro]{Definition}
\newtheorem{lem}[pro]{Lemma}
\newtheorem{cor}[pro]{Corollary}
\newtheorem{rk}[pro]{Remark}
\newtheorem{ex}[pro]{Example}

\newcommand{\Ext}{\mathrm{Ext}}
\newcommand{\Hom}{\mathrm{Hom}}
\newcommand{\End}{\mathrm{End}}

\newcommand{\T}{\mathcal{T}}
\newcommand{\C}{\mathcal{C}}
\newcommand{\X}{\mathcal{X}}
\newcommand{\Y}{\mathcal{Y}}
\newcommand{\Z}{\mathcal{Z}}

\newcommand{\pd}{\mathrm{pd}}
\newcommand{\inj}{\mathrm{inj}}

\newcommand{\gldim}{\mathrm{gl.dim}}

\newcommand{\ind}{\mathrm{ind}}
\newcommand{\rad}{\mathrm{rad}}

\newcommand{\add}{\mathrm{add}}
\newcommand{\smd}{\mathrm{smd}}

\newcommand{\gen}{\mathrm{gen}}
\newcommand{\Ker}{\mathrm{Ker}}
\newcommand{\Ima}{\mathrm{Im}}

\newcommand{\modu}{\mathrm{mod}}

\newcommand{\homoX}{\mathsf{K}^b(\X)}

\newcommand{\proj}{\mathrm{proj}}

\usepackage{latexsym,amssymb,amscd,rotating} 
\usepackage{amsmath}
\usepackage[all]{xypic}
\usepackage{amsthm}

\usepackage[Lenny]{fncychap}
\usepackage{amscd}

\begin{document}

\title[Relative Torsion Classes]{Relative Torsion Classes, relative tilting and relative silting modules}
\author{Luis Mart\'inez and Octavio Mendoza}
\thanks{2020 {\it{Mathematics Subject Classification}}. Primary 16G10; 18G05. Secondary 16E10.\\
{\it Keywords.} Auslander-Reiten theory; relative homological algebra; relative torsion classes; relative tilting and relative presilting modules.\\
The authors thanks project PAPIIT-Universidad Nacional Aut\'onoma de M\'exico IN100520.}

\begin{abstract} Let $\Lambda$ be an Artin algebra. In 2014, T. Adachi, O. Iyama and I. Reiten proved that the torsion funtorially finite classes in $\modu\,(\Lambda)$ can be described by the $\tau$-tilting theory.  The aim of this paper is to introduce the notion of $F$-torsion class in $\modu\,(\Lambda)$, where $F$ is an additive subfunctor of $\Ext^1_\Lambda,$ and to characterize when these clases are preenveloping and $F$-preenveloping. In order to do that, we introduce the notion of $F$-presilting $\Lambda$-module. The latter is both a generalization of $\tau$-rigid and $F$-tilting in $\modu(\Lambda).$
\end{abstract}  
\maketitle

\setcounter{tocdepth}{1}
\tableofcontents

\section{Introduction}

Throughout this paper $\Lambda$ will be an Artin $R$-algebra. We denote by $\modu(\Lambda)$ the category of finitely generated left $\Lambda$-modules and $\proj(\Lambda)$ (respectively, $\inj(\Lambda))$ denotes the category of finitely generated projective (respectively, injective) left $\Lambda$-modules. Following \cite{AR75}, we have the usual duality functor $D_\Lambda:\modu(\Lambda)\to \modu(\Lambda^{op}),$ where $D_\Lambda:=\Hom_R(-,k)$ and $k$ is the injective envelope of $R/\rad(R),$ the duality functor $*:=\Hom_\Lambda(-,\Lambda):\proj(\Lambda)\to\proj(\Lambda^{op})$ and the Nakayama equivalence $\nu:\proj(\Lambda)\to\inj(\Lambda),$ where $\nu(P):=D_{\Lambda^{op}}(P^*).$ For any $M\in\modu(\Lambda)$, we denote by $\gen(M)$ the epimorphic images of  finite direct sums of copies of $M,$ and by $\add(M)$ the class of all the direct summands of finite direct sums of copies of $M.$
\

The study of relative homological algebra over an Artin algebra $\Lambda$ is due to M. Auslander and \O. Solberg \cite{AS93, AS93II}. It has been very useful for studying relative tilting modules through a connection with partial tilting modules and his complemments \cite{AS93III}. In those papers, Auslander and Solberg found more general methods that allow us the construction of new subcategories that are contravariantly finite and functorially finite. It is also given a generalization to the relative context of the theory of tilting and cotilting modules giving new relationships between algebras.
\

All the subfunctors $F$ of $\Ext^1_\Lambda$ to be considered in this paper will be additive. We write $F\subseteq \Ext^1_\Lambda$ meaning that $F$ is an additive subfunctor of $\Ext^1_\Lambda.$
\

For a given  $F\subseteq \Ext^1_\Lambda,$ an exact sequence $\eta: 0\to A\to B\to C\to 0$ is called {\bf $F$-exact} if $\eta\in F(C,A).$
\

An {\bf additive generator} in $\modu(\Lambda)$ is a $\Lambda$-module $X$ such that $\gen(X)=\modu(\Lambda).$  Following  \cite{AS93}, for a given class $\X\subseteq\modu(\Lambda),$ there exists $F_\X\subseteq\Ext^1_\Lambda$ which is determined by all the exacts sequences $\eta: 0\to A\to B\to C\to0$ such that $\mathrm{Hom}_\Lambda(X,\eta)$ is an exact sequence $\forall$ $X\in\X.$
\

Let $\C$ be an additive category and $\X$ be a class of objects in $\C.$ Following \cite{AS80}, it is said that a morphism $f:X\rightarrow C$ is  an $\X$-{\bf{precover}} if $X\in\X$ and $\mathrm{Hom}_\C(Z,f):\mathrm{Hom}_\C(Z,X)\rightarrow\mathrm{Hom}_\C(Z,C)$ is surjective, for any $Z\in\X$. The class $\X$ is {\bf{precovering}} if every $C\in\C$ admits a $\X$-precover. Dually, we have the notion of $\X$-{\bf{preenvelope}} and {\bf preenveloping} class. If $\X$ is both precovering and preenveloping it is called {\bf functorially finite}.

In case that $\C=\modu(\Lambda),$ Auslander and Solberg defined, for  $F\subseteq\Ext^1_\Lambda,$ the notion of an {\bf $(\X,F)$-preenvelope} of $M,$ which is an $F$-exact sequence $\eta: 0\to M\overset{f}{\to}X\to N\to0$ such that $f:M\to X$ is an $\X$-preenvelope. It is said that $\X$ is {\bf $F$-preenveloping} if each $M\in\modu(\Lambda)$ admits an $(\X,F)$-preenvelope.
\

Tilting theory started as a module-theoretic interpretation of the reflection functors introduced by I. N. Berstein, I. M. Gelfand and V. A. Ponomarev in the early 1970s. These functors were categorized in $\modu(\Lambda)$ by Brenner and Butler \cite{BB80}, through the tilting modules. When $T$ is a tilting module, there is an associated torsion pair $(\mathcal{T},\mathcal{F})$, where $\mathcal{T}=\gen(T).$ Moreover,  for $\Gamma:=\mathrm{End}_\Lambda(T)^{op},$ we have the faithful functors $\mathrm{Hom}_\Lambda(T,-):\mathcal{T}\to\mathrm{mod}(\Gamma)$ and $\Ext^1_\Lambda(T,-):\mathcal{F}\to\mathrm{mod}(\Gamma).$ The above can be seen as a generalization of Morita equivalence \cite{Mo58}, which is recovered if $T$ is a progenerator. In the case that $\mathcal{T}$ is a funtorially finite torsion class, it is proved in \cite{AIR14}  that there is a support $\tau$-tilting module $M$ such that $\mathcal{T}=\gen(M),$ where $\tau$ is the Auslander-Reiten translation functor \cite{AR75}. Moreover, this assignment defines a bijection between functorially finite torsion classes and support 
$\tau$-tilting modules.
\

Consider a class $\X\subseteq\modu(\Lambda).$ The right orthogonal class of $\X$ is $\X^\perp:=\cap_{i>0}\,\X^{\perp_i},$ where $\X^{\perp_i}:=\{M\in\modu(\Lambda)\;|\;\Ext^i_\Lambda(-,M)|_{\X}=0\}$ is the $i$-th right orthogonal class of $\X,$ for each positive integer $i.$ Dually, we have the left orthogonal class ${}^{\perp}\X$ of $\X$ and the $i$-th left orthogonal class ${}^{\perp_i}\X$ of $\X.$ Following \cite{BB80}, it is said that $T\in\modu(\Lambda)$ is {\bf tilting} if  $\pd(T)\leq 1,$ $T\in T^\perp$ and there is an exact sequence $0\to {}_\Lambda\Lambda\to T_0\to T_1\to 0$ with $T_0,T_1\in\add(T).$
\

 Let $F\subseteq\Ext^1_\Lambda$ be with enough projectives and injectives (see definition in Section 2). In \cite{AS93}, it is defined the relative Ext bifunctor 
 $ \Ext^i_F(-,-),$  for any $i\geq 0,$ and the relative projective dimension $\pd_F(M),$ for any $M\in\modu(\Lambda).$   The {\bf $F$-global dimension} $\gldim_F(\Lambda)$ of 
 $\Lambda$
 is the supremum of all the $\pd_F(M),$ where $M\in\modu(\Lambda).$ 
 \
 
 For a given 
 $\X\subseteq\modu(\Lambda),$  the $F$-right orthogonal class $\X^{{}_F\perp}:=\cap_{i>0}\,\X^{{}_F\perp_i}$ of $\X,$   where 
 $\X^{{}_F\perp_i}:=\{M\in\modu(\Lambda)\;|\;\Ext^i_F(-,M)|_{\X}=0\}$ is the $i$-th $F$-right orthogonal class of $\X,$ for each positive integer $i.$ Following \cite{AS93II}, it is said that $T\in\modu(\Lambda)$ is {\bf F-tilting} if $\pd_F(T)\leq 1,$ $T\in T^{_F\perp}$ and, for any $Q\in\mathcal{P}(F),$  there is an $F$-exact sequence $0\to Q\to T_0\to T_1\to 0$ with $T_0,T_1\in\add(T).$ We recall that $\mathcal{P}(F)$ is the class of all the $F$-projective $\Lambda$-modules (see definition in Section 2).
\

As we have seen above, the case when $F\subseteq\Ext^1_\Lambda$ has enough projectives and injectives is quite relevant. Note that, if $X$ is an additive generator in $\modu(\Lambda),$ it follows from Lemma \ref{FP-FI} (a) that $F:=F_{\add(X)}$ has enough projectives and injectives. 
\

M. Hoshino  and S. Smal\o   ${}$ proved the following result characterizing the torsion preenveloping classes (containing  the injectives) 
through the tilting modules.

\begin{teo}\cite{Ho82, Sm84} For an Artin algebra $\Lambda$ and $\mathcal{T}\subseteq\modu(\Lambda),$ the following statements are equivalent.
\begin{itemize}\item[(a)] $\mathcal{T}$ is a preenveloping torsion class such that $\inj(\Lambda)\subseteq\mathcal{T}.$
\item[(b)] There is a tilting module $T\in\modu(\Lambda)$ such that $\mathcal{T}=\gen(T).$
\end{itemize}\end{teo}

The first main result of this paper is the relative version of the above theorem. In order to state such result, we need firstly to give the notions of $\gen_F(T)$ and $F$-torsion class, for some  $F\subseteq\Ext^1_\Lambda.$ In order to do that, for a class $\X\subseteq\modu(\Lambda),$ we denote by  $\X^\oplus $ the class of all the $\Lambda$-modules which are finite direct sums of copies of objects in   
$\X.$ 

\begin{defi} Let $\mathcal{T}\subseteq\mathrm{mod}(\Lambda)$ and $F\subseteq\Ext^1_\Lambda.$
\begin{itemize}
\item[$\mathrm{(a)}$] $g: L\to N$ $(respectively$, $f: M\to L)$ in $\modu(\Lambda)$ is {\bf F-epic} $(respectively,$ {\bf F-monic}$)$ if there is an $F$-exact sequence $0\to M\overset{f}{\to}L\overset{g}{\to}N\to0.$ In such a case, it is said that $N$ is an {\bf F-quotient} of $L$ $(respectively$, $M$ is an {\bf F-subobject} of $L).$
\item[$\mathrm{(b)}$] $\mathcal{T}$ is an {\bf F-torsion class} if $\mathcal{T}$ is closed under $F$-extensions and $F$-quotients.
\item[$\mathrm{(c)}$] $\gen_F(\mathcal{T})$ is the class of all the $Z\in\modu(\Lambda)$ for which there is an $F$-exact sequence $0\to K\to Y\to Z\to0$ with $Y\in\mathcal{\T}^\oplus.$
\end{itemize}\end{defi}

Now, we are ready to state the first main result of this paper which is the following one, for a proof see in Section 3. We recall that $M\in\modu(\Lambda)$ is {\bf basic} if $M\neq 0$ and for a decomposition into indecomposables $M=\oplus_{i=1}^n\,M_i$ we have that all the $M_1,M_2,\cdots, M_n$ are pairwise non isomorphic. 

\begin{teo}\label{NHo83}  For an additive generator $X$ in $\modu(\Lambda),$  $F:=F_{\mathrm{add}(X)}$ and 
$\mathcal{T}\subseteq\modu(\Lambda),$ the following statements are equivalent.
\begin{itemize}
\item[$\mathrm{(a)}$] The class $\mathcal{T}$  is $F$-preenveloping and  $F$-torsion. 
\item[$\mathrm{(b)}$] There is an $F$-tilting $T\in\modu(\Lambda)$ such that $\mathcal{T}=\mathrm{gen}_F(T).$
\end{itemize}
Moreover, the map
$T\mapsto\mathrm{gen}_F(T)$ is a bijection between the iso-classes of basic  $F$-tilting $\Lambda$-modules and the classes $\mathcal{T}\subseteq\modu(\Lambda)$ which are $F$-torsion and $F$-preenveloping.
\end{teo}

The concept of triangulated category, introduced by J. L. Verdier in \cite{Verdier}, has been very useful for studying the category $\modu(\Lambda).$ This is because certain triangulated categories become quite accessible if one applies the methods of representation theory  \cite{Ha88}. The most famous example \cite{Ha88} of a triangulated category is the derived category $\mathsf{D}^b(\modu(\Lambda))$ of bounded cochain complexes over $\modu(\Lambda)$ and can be identified with the homotopy category $\mathsf{K}^{-,b}(\proj(\Lambda))$ of bounded cochain complexes which are bounded above and have bounded cohomology over the class $\proj(\Lambda).$ Silting theory is a topic originating from representation theory of Artin algebras which is an extension of tilting theory in the environment of triangulated categories \cite{AMV16}.

We denote by $\mathsf{C}(\C)$ the category of cochain complexes of objects in and additive category $\C.$ The homotopy category of cochain complexes of objects in $\C$ is denoted by $\mathsf{K}(\C).$  The two-terms presilting complexes in $\mathsf{K}^b(\proj(\Lambda))$ are characterized by the $\tau$-tilting theory \cite{AIR14}. For an additive generator $X$ in $\modu(\Lambda)$ and $F:=F_{\add(X)},$ we give in  this paper  a classification of the two-terms presilting complexes in $\mathsf{K}^b(\add(X))$ by using the $\tau$-rigid modules in $\modu(\Gamma),$ where $\Gamma:=\mathrm{End}_\Lambda(X)^{op}.$

For each $M\in\modu(\Lambda),$ we fix a minimal $F$-projective resolution  of $M$
\begin{center}
$\cdots\to P^{-n}_F(M)\xrightarrow{\pi^{-n}_F}P^{-n+1}_F(M)\to\cdots\to P^{-1}_F(M)\xrightarrow{\pi^{-1}_F}P^0_F(M)\xrightarrow{\pi^0_F} M\to 0.$  
\end{center}
Associated to this minimal projective resolution of $M,$ we have the cochain complex $P^\bullet_F(M)\in\mathsf{C}(\modu(\Lambda))$ which agrees with this resolution in degree $i\leq 0$ and vanishes for any degree $i>0.$ In order to avoid doubt and underline the role of the $\Lambda$-module $M,$ in the resolution above, we some times write $\pi^{-n}_{F,M}$ instead of $\pi^{-n}_F.$ This will be usefull when considering another $\Lambda$-module $N$ and its minimal $F$-projective resolution.
\

For any cochain complex $W^\bullet\in\mathsf{C}(\modu(\Lambda))$ and any integer $n,$ we have the truncated cochain subcomplex $W^{\bullet}_{\geq n}$ of $W^{\bullet},$ where $W^i_{\geq n}:= W^i$ if $i\geq n$ and $W^{i}_{\geq n}:=0$ elsewhere. Similarly, we have the truncated quotient cochain complex  $W^{\bullet}_{\leq n}$ of $W^{\bullet}.$ 

Consider $M\in\modu(\Lambda).$ We denote by $\mathrm{rk}(M)$ the number of pairwise non-isomorphic indecomposable direct summands of $M.$ Following \cite{AIR14}, it is said that $M$ is {\bf $\tau$-rigid} if $\Hom(M,\tau(M))=0.$ Moreover, $M$ is called {\bf $\tau$-tilting} if it is $\tau$-rigid and $\mathrm{rk}(M)=\mathrm{rk}({}_\Lambda\Lambda).$ 

Let $\X$ be an additive subcategory of $\modu(\Lambda)$ and $\Z$ be a class of objects in the homotopy category $\homoX$ of bounded cochain complexes in $\X.$ For each positive integer $i,$ the right $i$-th orthogonal class of $\Z$ is 
$$\Z^{\perp_i}:=\{P^\bullet\in\homoX\;|\;\Hom_{\homoX}(-,P^\bullet[i])|_{\Z}=0\}.$$
we also consider the right orthogonal class $\Z^{\perp_{>i}}:=\cap_{j>i}\Z^{\perp_j}.$ Dually, we have the left orthogonal classes ${}^{\perp_{>i}}\Z$ and ${}^{\perp_i}\Z$ of $\Z.$ Following \cite{AI12}, we recall that a cochain complex $P^\bullet\in\homoX$ is {\bf presilting} if $P^\bullet\in(P^\bullet)^{\perp_{>0}}.$ Moreover, $P^\bullet$ is {\bf silting} if it is presilting and $\homoX$ is the smallest triangulated subcategory containing $\add(P^\bullet).$ In the case of a $\tau$-tilting $\Lambda$-module $M,$ the truncated complex $P^\bullet_{\geq -1}(M)$ is silting in $\mathsf{K}^b(\proj(\Lambda))$ \cite[Proposition 3.7 (a)]{AIR14}.

\begin{defi}Let $X$ be an additive generator in $\mathrm{mod}(\Lambda)$ and $F:=F_{\mathrm{add}(X)}.$ We say that $M\in\modu(\Lambda)$ is {\bf F-silting} $(${\bf F-presilting}$)$ if $\mathrm{P^\bullet_{F\geq-1}}(M)$ is silting $(presilting)$ in $\mathsf{K}^b(\mathrm{add}(X)).$
\end{defi}

The next result, which is proved in Section 4,  is a generalization of \cite[Lemma 3.4]{AIR14}.

\begin{pro}\label{RAIR, 3.5} Let $X$ be an additive generator in $\modu(\Lambda),$  $F:=F_{\mathrm{add}(X)}$ and $\Gamma:=\End_\Lambda(X)^{op}.$ Then, for any $M\in\modu(\Lambda),$ the following  
 statements are equivalent.
\begin{itemize}
\item[$\mathrm{(a)}$] $M$ is $F$-presilting.
\item[$\mathrm{(b)}$] $\Hom_\Lambda(\pi^{-2}_{F,M}, M)=0$ and $\Ext^1_F(M,M)=0.$
\item[$\mathrm{(c)}$] $\Hom_\Lambda(\pi^{-2}_{F,M}, M)=0$ and $M\in{}^{_F\perp_1}\gen_F(M).$
\item[$\mathrm{(d)}$] $\mathrm{Hom}_\Lambda(X,M)$ is $\tau$-rigid in $\mathrm{mod}(\Gamma).$
\end{itemize} 
\end{pro}

Note that the notion of $F$-presilting module in $\modu(\Lambda)$ coincides with that of $\tau$-rigid, when $X=\Lambda.$ Moreover, it is a generalization of $F$-tilting by the above equivalences.  If $M$ is a $\tau$-rigid module, we know that  $\gen(M)$ is a preenveloping class \cite[Theorem 5.10]{AS81}. If $M$ is an $F$-presilting module, then 
$\gen_F(M)$ is not necessarily a preenveloping class as can be see in Example $\ref{Kronecker}.$
\

A natural question that arrives is when the condition $\Hom_\Lambda(\pi^{-2}_{F,M}, M)=0$ holds true? In what follows, we see three cases where this is possible.

\begin{rk}  Let $X$ be an additive generator in $\modu(\Lambda)$ and $F:=F_{\mathrm{add}(X)}.$ Then, for $M\in\modu(\Lambda),$ we have:
\begin{enumerate}
\item If $X=\Lambda$ and $M\in{}^{_F\perp_1}\gen_F(M),$ then by \cite[Proposition 1.2 (a), Proposition 2.4 (b)]{AIR14} we get that $\Hom_\Lambda(\pi^{-2}_{F,M}, M)=0.$ 
\item If $\pd_F(M)\leq 1$ then $\Hom_\Lambda(\pi^{-2}_{F,M}, M)=0.$ 
\item If $X$ is an Auslander generator in $\modu(\Lambda)$ and $\mathrm{rep.dim}(\Lambda)\leq 3,$ then $\Hom_\Lambda(\pi^{-2}_{F,M}, M)=0.$ Indeed, it follows from 
$\mathrm{(2)}$  and the discussion in \cite[pag. 785]{LM17}.
\end{enumerate}
\end{rk} 

The discussion in the above remark motivates the following definition.

\begin{defi}  Let $X$ be an additive generator in $\modu(\Lambda)$ and  $F:=F_{\mathrm{add}(X)}.$ We say that the algebra $\Lambda$ is {\bf $F$-admissible} if, for 
any $M\in\modu(\Lambda),$  we have that $\Ext^1_F(M,M)=0$ implies that $\Hom_\Lambda(\pi^{-2}_{F,M}, M)=0.$ 
\end{defi} 

In Example \ref{Ej-noFadm}, we give an Artin algebra $\Lambda$ and an additive generator $X\in\modu(\Lambda)$ such that $\Lambda$ is not $F$-admissible, for 
$F:=F_{\mathrm{add}(X)}.$ 
\

The second main result of this paper is a generalization of the bijection between preenveloping torsion class in 
$\modu(\Lambda)$ and basic support $\tau$-tilting $\Lambda$-modules \cite[Theorem 2.7]{AIR14}. In order to state that, we need the following notions. Let 
$\T\subseteq\modu(\Lambda)$ be such that $\T=\add(\T)$ and $\ind(\T\cap{}^{\perp_1}\T)=\{T_i\}_{i=1}^n.$ We set $\mathbb{P}(\T):=\oplus_{i=1}^n\,T_i.$ An example when this happens is the case when $\T=\gen(M)$ for some $\tau$-rigid $M\in\modu(\Lambda)$ \cite[Corollary 4.4]{AS81}. For $M,N\in\modu(\Lambda),$ the notation $N\mid M$ means that $N$ is a direct summand of $M.$

\begin{defi} Let $X$ be an additive generator in $\modu(\Lambda),$ $F:=F_{\mathrm{add}(X)}$ and $M\in\modu(\Lambda).$ We say that $M$ is {\bf special F-presilting} if $M$ is basic, $F$-presilting and $\mathrm{Hom}_\Lambda(X,M\oplus M')\nmid \mathbb{P}(\mathrm{gen(Hom}_\Lambda(X,M)))$ $\forall\,M'\in\modu(\Lambda)-\{0\}$ such that $\mathrm{add}(M)\cap\mathrm{add}(M')=0.$ 
\end{defi}

Note that, in the above definition, $\mathbb{P}(\mathrm{gen(Hom}_\Lambda(X,M))))$ exists since by Proposition \ref{RAIR, 3.5} we know that $\mathrm{Hom}_\Lambda(X,M)$ is $\tau$-rigid in $\mathrm{mod}(\Gamma).$ Now, we are ready to state the second main result of this paper which is the following one. For a proof, see in Section 4.

\begin{teo}\label{teo-princip2} Let $X$ be an additive generator in $\modu(\Lambda)$ and $F:=F_{\mathrm{add}(X)}$ be such that $\Lambda$ is $F$-admissible.  Consider the following classes:
\begin{itemize}
\item[$\mathrm{(a)}$] $\mathrm{F}$-$\mathrm{tor}$-$\mathrm{preen}(\Lambda)$ whose elements are the nonzero preenveloping $F$-torsion classes in $\modu(\Lambda);$
\item[$\mathrm{(b)}$] $\mathrm{F}$-$\mathrm{presilt}$-$\mathrm{esp}(\Lambda)$ whose elements are the iso-classes $[M]$, with $M\in\modu(\Lambda)$  special F-presilting such that $\gen_F(M)$ is preenveloping.
\end{itemize}
Then, the map $[M]\mapsto\mathrm{gen}_F(M)$ induces a bijection between $\mathrm{F}$-$\mathrm{presilt}$-$\mathrm{esp}(\Lambda)$ and $\mathrm{F}$-$\mathrm{tor}$-$\mathrm{preen}(\Lambda).$
\end{teo}

Finally, in Example \ref{Ejem4}, we give an Artin algebra $\Lambda$ where we can apply Theorem \ref{teo-princip2} but not the Theorem \ref{NHo83}.

\section{Background and preliminary results}

We start this section by collecting all the background material that will be necessary in the sequel. In all that follows,  $\Lambda$ denotes an Artin $R$-algebra.
\

Let $M\in\modu(\Lambda).$ The following technique allows solving problems in $\add(M)$ and takes them into questions about projective 
$\Gamma$-modules, where $\Gamma:=\mathrm{End}_\Lambda(M)^{op}.$

\begin{pro}\cite[Proposition II.2.1]{ARS95} \label{ASS06, VI.3.1}  For $M\in\modu(\Lambda),$  $\Gamma:=\mathrm{End}_\Lambda(M)^{op}$ and the evaluation functor $e_M:=\mathrm{Hom}_\Lambda(M,$-$): \modu(\Lambda)\to\modu(\Gamma),$ the following statements hold true.
\begin{itemize}
\item[(a)]  $e_M:\Hom_\Lambda(Z,X)\to \Hom_\Gamma(e_M(Z), e_M(X))$ is an isomorphism in $\modu(\Gamma),$ $\forall\,Z\in\add(M);$ $\forall\,X\in\modu(\Lambda).$ 
\item[(b)] $e_M|_{\add(M)}: \add(M)\to\proj(\Gamma)$ is an equivalence of categories.
\end{itemize}
\end{pro}

Let $F\subseteq\Ext^1_\Lambda.$ The class $\mathcal{I}(F)$ of the {\bf F-injective} modules  consists of all the 
$I\in\modu(\Lambda)$ such that, for every $F$-exact sequence $\eta: 0\to M\to N\to K\to0,$ the sequence $\Hom_\Lambda(\eta, I)$ is exact.  it is said that $F$ has {\bf enough injectives} if any $M\in\modu(\Lambda)$ admits an $F$-exact sequence $0\to M\to I\to Q\to 0$ with 
$I\in \mathcal{I}(F).$ Dually, it is defined the class $\mathcal{P}(F)$ of the  {\bf F-projective} modules and the notion saying that $F$ has {\bf enough projectives}.
\

 The next lemma will be  crucial for the development of the paper.

\begin{lem}\label{Sa12, 1.10} For $F\subseteq\Ext^1_\Lambda,$ the following statements hold true.
\begin{itemize}
\item[(a)] The class of all the $F$-exact sequences  is closed under pushouts, pulbacks and finite direct sums.
\item[(b)] Let $\eta: 0\to M\to N\to K\to0$ and $\eta': 0\to M'\to N'\to K'\to0$ be exact sequences in $\modu(\Lambda).$ Then,   
$\eta\oplus\eta'$ is $F$-exact if, and only if, $\eta$ and $\eta'$ are $F$-exact.
\item[(c)] $\gen_F(\Y)=\gen_F(\add(\Y)),$ for any $\Y\subseteq \modu(\Lambda).$
\item[(d)] Let $F=F_\X,$ for some $\X\subseteq\modu(\Lambda).$ Consider the following exact and commutative diagram in $\modu(\Lambda)$ 
$$\xymatrix{&0\ar[d]&0\ar[d]\\
&A\ar[d]\ar@{=}[r]&A\ar[d]\\
0\ar[r]& B\ar[r]\ar[d] & B' \ar[d] \ar[r]& B''\ar@{=}[d]\ar[r]&0\\
0\ar[r]& C\ar[d]\ar[r]& C'\ar[d]\ar[r]&C''\ar[r]&0\\
&0&0.}$$
If the middle row and the first column in the above diagram are $F$-exact, then all the rows and columns are $F$-exact.
\item[(e)] Let $F$ be with enough projectives and $\eta:\;0\to A\to B\to C\to 0$  be an exact sequence in $\modu(\Lambda).$ Then, $\eta$ is $F$-exact if, and only if, $\Hom_\Lambda(P,\eta)$ is exact for all $P\in\mathcal{P}(F).$
\item[(f)] Let $F$ be with enough projectives and $f:M\to N$ and $g:N\to K$ be morphisms in $\modu(\Lambda).$ Then, the following statements hold true.
 \begin{itemize}
 \item[(f1)] If $gf$ is  $F$-epic, then $g$ is $F$-epic.
 \item[(f2)] If $f$ and $g$ are $F$-epic, then $gf$ is $F$-epic.
 \end{itemize}

\end{itemize}
\end{lem}
\begin{proof} The proof of (a) and (b) can be found in \cite{AS93}, whereas the proof of (d) can be obtained from 
\cite[Proposition 3.9]{S19}. Moreover, (e) follows from \cite[Proposition 1.5 (a)]{AS93}.
\

Let us prove (c). Consider $\Y\subseteq \modu(\Lambda).$ It is clear that  $\gen_F(\Y)\subseteq\gen_F(\add(\Y)).$ Let 
$M\in \gen_F(\add(\Y)).$ Then there is an $F$-exact sequence $0\to K\to\Y'\to M\to 0$ such that $Y'\oplus Y''=Y^m$ and $Y\in \Y.$ Since 
$0\to Y''\to Y''\to 0\to 0$  is $F$-exact, we get from (a) that $0\to K\oplus Y''\to Y^m\to M\to 0$ is $F$-exact and thus $M\in\gen_F(\Y).$
Finally, (f) can be proved by using  (e).
\end{proof}

For a given class $\X\subseteq\modu(\Lambda),$ we denote by $\smd(\X)$ the class of all the $\Lambda$-modules which are a direct summand of objects in $\X.$

\begin{lem} \label{FP-FI} For $F:=F_{\add(X)},$ where $X$ is an additive generator in $\modu(\Lambda),$ the following statements hold true.
\begin{itemize}
\item[(a)] $F$ has enough injectives and projectives.
\item[(b)] $\mathcal{P}(F)=\add(X)$ and $\mathcal{I}(F)=\add(\tau(X))\oplus\inj(\Lambda).$
\item[(c)] Let $\X=\smd(\X)\subseteq\modu(\Lambda).$ Then, $\X$  is $F$-preenveloping  if, and only if, $\X$ is preenveloping  and 
$\mathcal{I}(F)\subseteq\X.$
\item[(d)] Let $M\in\modu(\Lambda).$ Then, $\pd_F(M)\leq n$ $\Leftrightarrow$ $\Ext^{n+1}_F(M,-)=0.$ 
\end{itemize}
\end{lem}
\begin{proof} The item (a)  follows from \cite[Proposition 1.10 and Corollary 1.13]{AS93}, and  (b) can be obtained from  \cite[Proposition 1.10 and Corollary 1.6 (a)]{AS93}. 
 \
 
Let us prove (c). By (b),  $\mathcal{P}(F)=\add(X).$ Then, (c)  follows from \cite[Proposition 4.2]{AS93}. In order to prove (d), let $\Ext^{n+1}_F(M,-)=0.$ Consider $N\in\modu(\Lambda).$ Since, by (a), $F$-has enough injectives, there is an $F$-exact sequence 
$0\to N\to I_0\to I_1\to\cdots\to I_{j-1}\to K\to 0$ where $I_k\in\mathcal{I}(F)$ $\forall\,k.$ Hence 
$\Ext^{n+j+1}_F(M,N)\simeq \Ext^{n+1}_F(M,K)=0$ $\forall\, j\geq 1$ and thus $\pd_F(M)\leq n.$
\end{proof}

Following G. M. Kelly in \cite{Ke64}, we recall the notion the {\bf radical} ideal $\mathrm{rad}_\Lambda$ of the category 
$\modu(\Lambda).$ This ideal is defined, for any $M,N\in\modu(\Lambda),$ as follows 
$$\mathrm{rad}_\Lambda(M,N):=\{f\in\Hom_\Lambda( M,N)\;|\;1_M-gf\in U(\End_\Lambda(M))\;\forall\,g\in\Hom_\Lambda(N,M)\}.$$
 The notion of right and left minimal morphisms were introduced, in \cite{AS80}, for the study of the  preprojective partitions in 
 $\modu(\Lambda).$ We recall that  $f: M\to N$ in $\modu(\Lambda)$ is called {\bf left minimal} if every morphism $g: N\to N$ such that $gf=f$  is an isomorphism. Dually, it is defined {\bf rigt minimal} morphism. The following result shows us the relationship between the previous concepts.

\begin{lem}\cite[Lemma 2.5]{ABM98}\label{ABM98, 2.5} For a morphism $f: M\to N$  in $\modu(\Lambda),$  the following statements are equivalent.
\begin{itemize}
\item[$(a)$] $f$ is left minimal.
\item[$(b)$] For all pseudocokernel $g:N\to K$ of $f,$ we have that  $g\in\rad_\Lambda(N,K).$ 
\item[$(c)$] There is some  pseudocokernel $g: N\to K$  of $f$ such that $g\in\rad_\Lambda(N,K).$
\end{itemize}
\end{lem}

\section{$F$-torsion and $F$-preenveloping classes in $\modu(\Lambda)$}

In this section, we study the $F$-torsion classes in $\modu(\Lambda)$ and their relationship with the $F$-tilting $\Lambda$-modules.

\begin{lem}\label{T=addT} Let $\mathcal{T}\subseteq\modu(\Lambda)$ and $F\subseteq\Ext^1_\Lambda.$ If 
$\mathcal{T}$ is an $F$-torsion class, then $\mathcal{T}=\mathrm{add}(\mathcal{T}).$
\begin{proof} Let $\T$ be $F$-torsion. Since every split exact sequence is $F$-exact, from the exact sequence 
$0\to T\overset{1_T}{\to}T\to0\to0$, with $T\in\mathcal{T},$  it follows that $0\in\mathcal{T}.$ Now, since $\mathcal{T}$ closed under $F$-extensions and $F$-quotients, we get that $\mathrm{add}(\mathcal{T})=\mathcal{T}$.
\end{proof}
\end{lem}

\begin{lem}\label{l2, NAS, 4.6} Let $\mathcal{X}\subseteq\mathrm{mod}(\Lambda)$, $M\in\mathrm{mod}(\Lambda)$, $F\subseteq\Ext^1_\Lambda$ and $f: M\to X$ be an $\mathcal{X}$-preenvelope. If $M$ is an $F$-subobject of some $X'\in\mathcal{X},$ then $f$ is $F$-monic.
\begin{proof} Since $f: M\to X$ is an $\mathcal{X}$-preenvelope and $M$ is a $F$-subobjet of $X'\in\mathcal{X}$, we have the following exact and commutative diagram in $\modu(\Lambda)$
$$\xymatrix{\eta:\;0\ar[r]&M\ar[r]^{f}\ar@{=}[d]&X\ar[r]\ar[d]&K'\ar[r]\ar[d]&0\\
\eta': \;0\ar[r]&M\ar[r]&X'\ar[r]&K\ar[r]&0,}$$
where $\eta'$ is $F$-exact. Thus, $\eta$ is a pull-back from $\eta'$ and by Lemma \ref{Sa12, 1.10} (a), it follows that 
 $0\to M\overset{f}{\to}X\to K'\to0$ is $F$-exact.
 \end{proof}
 \end{lem}

The next result is the relative version of \cite[Proposition 4.6]{AS80}.

\begin{pro}\label{NAS80, 4.6} For an additive generator $X\in\mathrm{mod}(\Lambda)$, $F:=F_{\mathrm{add}(X)}$ and $\mathcal{T}\subseteq\mathrm{mod}(\Lambda)$, the following statements hold true.
\begin{itemize}
\item[(a)] Let $\mathcal{T}$ be an $F$-torsion and  preenveloping class. Then $\mathcal{T}=\mathrm{gen}_F(T)$, where $f: X\to T$ is a $\mathcal{T}$-preenvelope of $X.$
\item[(b)] Let $\mathcal{I}(F)\subseteq\mathcal{T}=\mathrm{gen}_F(T)$, for some $T\in\mathcal{T}$. Then, the following conditions are satisfied.
  \begin{itemize}
  \item[(b1)] $\forall$ $X'\in\mathrm{add}(X)$ there is a $(\mathcal{T},F)$-preenvelope $f: X'\to T'$, with $T'\in\mathrm{add}(T)$.
 \item[ (b2)] $\mathcal{T}$ is $F$-preenveloping.
 \end{itemize}
\end{itemize}
\begin{proof} (a) Let $f: X\to T$ be a $\mathcal{T}$-preenvelope of $X.$ We assert that $\mathcal{T}=\mathrm{gen}_F(T).$ Indeed, the inclusion $\mathrm{gen}_F(T)\subseteq\T$ is clear since $\mathcal{T}$ is closed under finite coproducts and $F$-quotients.
\

 Let $T'\in\mathcal{T}$. By Lemma \ref{FP-FI} (a,b), there is an $F$-exact sequence 
 $\eta: 0\to K \to X'\to T'\to 0,$ with $X'\oplus X''=X^m.$ Since $\xi: 0\to X ''\to X ''\to0\to0$ is $F$-exact, it follows from Lemma \ref{Sa12, 1.10} (a) that $\eta\oplus \xi:\;0\to K\oplus X''\to X^m\to T'\to0$ is  $F$-exact. Using that $f^m:X^m\to T^m$ is a $\mathcal{T}$-preenvelope, we can get the following exact commutative diagram in $\modu(\Lambda)$
$$\xymatrix{0\ar[r]&K\oplus X''\ar[r]\ar[d]&X^m\ar[r]\ar[d]^{f^m}&T'\ar[r]\ar@{=}[d]&0\\
0\ar[r]&K'\ar[r]&T^m\ar[r]_{h}&T'\ar[r]&0.}$$
Thus the bottom row of the above diagram  is the pushout of the top one and hence, by Lemma \ref{Sa12, 1.10} (a),  $h:T^m\to T '$ is 
$F$-epic; proving that $\mathcal{T}\subseteq\mathrm{gen}_F(T).$
\

(b) Let $X'\in\mathrm{add}(X)$ and $g: X'\to T'$ be an $\mathrm{add}(T)$-preenvelope. To prove (b1), we show that $g: X'\to T'$ is a 
$(\mathcal{T},F)$-preenvelope. Indeed, let $Z\in\T=\mathrm{gen}_F(T)$, $k: T^m\to Z$ be $F$-epic and $h: X'\to Z.$ Since $X'$ is $F$-projective, there exists $l: X'\to T^m$ such that $h=kl.$ As $g: X'\to T'$ is an $\mathrm{add}(T)$-preenvelope, there is $r: T'\to T^m$ such that $rg=l$. Then $krg=kl=h$ and thus $g$ is a $\T$-preenvelope. On the other hand, by Lemma \ref{FP-FI} (a), there is a 
$F$-exact sequence $0\to X'\to I\to K\to0$, with $I\in\mathcal{I}(F)\subseteq\mathcal{T}$. Hence, by Lemma \ref{l2, NAS, 4.6}, we conclude that $g: X'\to T'$ is a $(\mathcal{T},F)$-preenvelope.
\

Let us show now that $\T$ is $F$-preenveloping. Indeed, let $N\in\mathrm{mod}(\Lambda).$ By Lemma \ref{FP-FI} (a,b), there is an 
$F$-exact sequence $0\to A\to X'\overset{p}{\to}N\to0$ with $X'\in\mathrm{add}(X).$ On the other hand,  by (b1) there is an 
$(\mathcal{T},F)$-preenvelope $0\to X'\overset{w}{\to}\bar{T}\to C\to0$, with $\bar{T}\in\mathrm{add}(T)$. Then, by Lemma \ref{Sa12, 1.10} (d),  we obtain the following $F$-exact and commutative diagram
$$\xymatrix{&0\ar[d]&0\ar[d]\\
&A\ar[d]\ar@{=}[r]&A\ar[d]\\
0\ar[r]&X'\ar[r]^w\ar[d]_{p}&\bar{T}\ar[d]^{p'}\ar[r]&C\ar@{=}[d]\ar[r]&0\\
0\ar[r]&N\ar[d]\ar[r]_{w'}&Z\ar[d]\ar[r]&C\ar[r]&0\\
&0&0.}$$
In particular,  $0\to N\overset{w'}{\to}Z\to C\to0$ and $0\to A\to\bar{T}\overset{p'}{\to}Z\to0$ are $F$-exact sequences.  Hence, by Lemma \ref{Sa12, 1.10} (c), it follows that $Z\in\mathrm{gen}_F(\mathrm{add}(T))=\mathrm{gen}_F(T)=\mathcal{T}.$ Let us show that  $w': N\to Z$ is a $\mathcal{T}$-preenvelope. Indeed, let $g: N\to Y$ with $Y\in\mathcal{T}.$ Since $w: X'\to\bar{T}$ is a $\mathcal{T}$-preenvelope, there exists $t: \bar{T}\to Y$ such that $tw=gp.$ By the universal property of the pushout, the exists $w':Z\to Y$ such that the 
following diagram commutes
$$\xymatrix{X'\ar[r]^w\ar[d]_{p}&\bar{T}\ar[d]^{p'}\ar@/^4pc/[ddl]^t\\
N\ar[r]_{w'}\ar[d]_{g}&Z\ar@{-->}[ld]^{\exists\;g'}\\
Y.}$$
Therefore $w': N\to Z$ is a $(\mathcal{T},F)$-preenvelope, and thus, $\mathcal{T}$ is $F$-preenveloping. 
\end{proof}
\end{pro}

The next result is the relative version of \cite[Theorem 5.10]{AS81}.

\begin{lem}\label{NAS81, 5.10} Let $X,$ $M\in\modu(\Lambda)$ be such that $X$ is an adittive generator, $F:=F_{\mathrm{add}(X)}$ and $M\in{}^{_F\perp_1}\mathrm{gen}_F(M).$ Then, $\mathrm{gen}_F(M)$ is an $F$-torsion class.

\begin{proof} We start by proving that  $\mathrm{gen}_F(M)$ is closed under $F$-quotients. Let $p: L\to L'$ be  $F$-epic, with $L\in\mathrm{gen}_F(M).$ In particular, there is an $F$-epic $q: M^r\to L.$ Then, by Lemma \ref{Sa12, 1.10} (f2),  $pq: M^r\to L'$ is $F$-epic and thus $\mathrm{gen}_F(M)$ is closed under $F$-quotients.
\

Let us show that  $\mathrm{gen}_F(M)$ is closed under $F$-extensions. Indeed, consider an $F$-exact sequence 
$\eta: 0\to A\overset{f}{\to}B\overset{g}{\to}C\to0$  with $A$ and $C$ in $\mathrm{gen}_F(M).$ Let $h: M'\to B$ be an $\mathrm{add}(M)$-precover. We assert that $h$ is  $F$-epic. By Lemma \ref{FP-FI} (a,b), there is an $F$-epic $k: X'\to B,$ with $X'\in\mathrm{add}(X).$ On the other hand, since $C\in\mathrm{gen}_F(M)$ there is an $F$-epic $h': M^n\to C.$  Therefore, there exists $k': X'\to M^n$ such that $h'k'=gk.$ By applying $\mathrm{Hom}_\Lambda(M^n,-)$ to $\eta,$ and since $\mathrm{add}(M)\in{}^{_F\perp_1}\mathrm{gen}_F(M),$ we get the exact sequence 
$$\mathrm{Hom}_\Lambda(M^n,B)\xrightarrow{\mathrm{Hom}_\Lambda(M^n,g)}\mathrm{Hom}_\Lambda(M^n,C)\to0.$$
Then, there is $h'': M^n\to B$ such that $h'=gh''.$ Now $g(h''k'-k)=h'k'-gk=0,$ and due to the universal property of the kernel, we have the following commutative diagram
$$\xymatrix{&&X'\ar@{-->}[ld]_{\exists!\;t}\ar[d]^{h''k'-k}\\
0\ar[r]&A\ar[r]_{f}&B\ar[r]_{g}&C\ar[r]&0.}$$
Since $A\in\mathrm{gen}_F(M),$ there is an $F$-epic $l: M^m\to A.$ Then, by Lemma \ref{FP-FI} (b),  there is $s: X'\to M^m$ such that $t=ls.$ Since $h: M'\to B$ is an $\mathrm{add}(M)$-precover, exists $q: M^m\to M'$ such that $fl=hq.$ Finally, since $h: M'\to B$ is an $\mathrm{add}(M)$-precover, there is $q': M^n\to M'$ such that $h''=hq'.$ So that $h(q'k'-qs)=h''k'-hqs=h''k'-ft=k$ and $h$ is an epimorphism. Since $k$ is  $F$-epic, from the equality $h(q'k'-qs)=k$ and Lemma \ref{Sa12, 1.10} (f1), it follows that $h$ is $F$-epic. Then $B\in\mathrm{gen}_F(M)$ and thus $\mathrm{gen}_F(M)$ is closed under $F$-extensions.
\end{proof}
\end{lem}

\begin{defi}Let $M\in\modu(\Lambda)$. We say that $M$ is $\mathbf{gen_F}$-$\mathbf{minimal}$ if $M\neq0$ and $\forall$ $M'\;|\;M$ $(M'\in\mathrm{gen}_F(M/M')\Rightarrow M'=0)$.\end{defi}

\begin{pro}\label{p1.1, NAS81, 5.1}Let $M$, $X\in\modu(\Lambda)$, with $M\neq0$, such that $X$ is an additive generator and $F:=F_{\mathrm{add}(X)}$. Then, there exists $M'\;|\;M$ such that $M'$ is $\mathrm{gen}_F$-minimal and $\mathrm{gen}_F(M')=\mathrm{gen}_F(M).$
\begin{proof} The  proof will be carried on by induction on $n:=\ell(M),$ the length of $M.$  The case $n=1$ is clear. 
\

let $n>1$. If $M$ is gen$_F$-minimal there is nothing to prove. Thus, we can assume that $M$ is not $\mathrm{gen}_F$-minimal. Then, there exists $M'\neq0$ such that $M'\;|\;M$ and $M'\in\mathrm{gen}_F(M/M').$ Since $\ell(M')<\ell(M),$ by the inductive hypothesis, there is  some $Y\;|\;(M/M')$ which is $\mathrm{gen}_F$-minimal and $\mathrm{gen}_F(Y)=\mathrm{gen}_F(M/M').$
\

Let us show that $\mathrm{gen}_F(Y)\subseteq\mathrm{gen}_F(M).$ Let $Z\in \gen_F(Y).$ Then, there is an $F$-exact sequence 
$\eta:\;0\to K\to Y^m\to Z\to 0.$ Since $Y\,|\,M,$ there is some $Y_1$ such that $Y\oplus Y_1=M^n.$ Consider the $F$-exact sequence 
$\eta':\;0\to Y_1^m\to Y_1^m\to 0\to 0.$ Then, Lema \ref{Sa12, 1.10} (b), we get that $\eta\oplus\eta'$ is an $F$- exact sequence and thus $Z\in \gen_F(M).$
\

We prove that $\mathrm{gen}_F(M)\subseteq\gen_F(Y).$ Indeed, let $Z\in\gen_F(M)$. Then there are
 $F$-epics $(f,g): (M/M')^m\oplus (M')^m\to Z$ and $h: (M/M')^s\to(M')^m$. On the other hand, by Lemma \ref{FP-FI} (a,b), there is an $F$-epic $k:X'\to Z,$ with $X'\in\add(X)=\mathcal{P}(F).$ Since $(f,g) $ and $h$ are $F$-epics, there are $\begin{pmatrix} f'\\ g' \end{pmatrix}: X'\to (M/M')^m\oplus (M')^m$ such that $(f,g)\begin{pmatrix} f'\\ g' \end{pmatrix}=k$ and $h': X'\to (M/M')^s$ such that $hh'=g'.$ Hence $k=ff'+gg'=ff'+ghh'=(f,gh)\begin{pmatrix} f'\\ h' \end{pmatrix},$ and by Lemma \ref{Sa12, 1.10} (f1),  it follows that $(f,gh): (M/M')^m\oplus(M/M')^s\to Z$ is  $F$-epic. Thus, by Lemma \ref{Sa12, 1.10} (c),  $Z\in\mathrm{gen}_F(\mathrm{add}(M/M'))=\mathrm{gen}_F(M/M')$ and then $Z\in\mathrm{gen}_F(M/M')=\mathrm{gen}_F(Y).$
\end{proof}
\end{pro}

The following lemma will be very useful in the development of this paper.

\begin{lem}\label{l2, NAS81, 5.1} Let $X$, $N\in\mathrm{mod}(\Lambda)$, with $X$ an additive generator, $\Gamma:=\mathrm{End}_\Lambda(X)^{op}$ and $F:=F_{\mathrm{add}(X)}.$ Then, the following statements hold true.
\begin{itemize}
\item[(a)] $\mathrm{Hom}_\Lambda(X,-): \mathrm{mod}(\Lambda)\to\mathrm{mod}(\Gamma)$ is full and faithful.
\item[(b)] $\mathrm{Ext^i_F}(-,N)\cong\mathrm{Ext^i_\Gamma}(\mathrm{Hom}_\Lambda(X,-),\mathrm{Hom}_\Lambda(X,N))$ $\forall$ $i\geq1.$
\end{itemize}
\end{lem}

\begin{proof} Let $M\in\mathrm{mod}(\Lambda)$  and $X_\bullet: \cdots\to X_1\overset{f_1}{\to}X_0\overset{f_0}{\to}M\to0$ be an  $F$-projective resolution  of $M,$ with $\{X_i\}_{i\in\mathbb{N}}\subseteq\add(X)$ (see Lemma \ref{FP-FI} (a,b)). In particular, we have the following chain complex, which is acyclic in the first two places,
$$0\to\mathrm{Hom}_\Lambda(M,N)\xrightarrow{\mathrm{Hom}_\Lambda(f_0,N)}\mathrm{Hom}_\Lambda(X_0,N)\xrightarrow{\mathrm{Hom}_\Lambda(f_1,N)}\mathrm{Hom}_\Lambda(X_1,N)\to\cdots.$$
By Proposition \ref{ASS06, VI.3.1}, we have the projective resolution of $\mathrm{Hom}_\Lambda(X,M)$ in $\modu(\Gamma)$
$$\cdots\to\mathrm{Hom}_\Lambda(X,X_1)\to\mathrm{Hom}_\Lambda(X,X_0)\to\mathrm{Hom}_\Lambda(X,M)\to0.$$

(a) We have the following exact and commutative diagram, where $\eta_{X_0}$ and $\eta_{X_1}$ are isomorphisms (see Proposition \ref{ASS06, VI.3.1})
$$\xymatrix{0\ar[d]&&&0\ar[d]\\
\mathrm{Hom}_\Lambda(M,N)\ar[rrr]^{\eta_N}\ar[d]_{\mathrm{Hom}_\Lambda(f_0,N)}&&&\mathrm{Hom}_\Gamma(\mathrm{Hom}_\Lambda(X,M),\mathrm{Hom}_\Lambda(X,N))\ar[d]_{\mathrm{Hom}_\Gamma(\mathrm{Hom}_\Lambda(X,f_0),\mathrm{Hom}_\Lambda(X,N))}\\
\mathrm{Hom}_\Lambda(X_0,N)\ar[rrr]^\sim_{\eta_{X_0}}\ar[d]_{\mathrm{Hom}_\Lambda(f_1,N)}&&&\mathrm{Hom}_\Gamma(\mathrm{Hom}_\Lambda(X,X_0),\mathrm{Hom}_\Lambda(X,N))\ar[d]_{\mathrm{Hom}_\Gamma(\mathrm{Hom}_\Lambda(X,f_1),\mathrm{Hom}_\Lambda(X,N))}\\
\mathrm{Hom}_\Lambda(X_1,N)\ar[rrr]^\sim_{\eta_{X_1}}&&&\mathrm{Hom}_\Gamma(\mathrm{Hom}_\Lambda(X,X_1),\mathrm{Hom}_\Lambda(X,N)).}$$
Then, by  Five Lemma,  it follows that $\eta_N$ is an isomorphism and thus $\mathrm{Hom}_\Lambda(X,-): \mathrm{mod}(\Lambda)\to\mathrm{mod}(\Gamma)$ is full and faithful.

(b) In the following commutative diagram, by Proposition \ref{ASS06, VI.3.1}, we have that the horizontal maps are isomorphisms $\forall$ $i\in\mathbb{N}$
$$\xymatrix{\mathrm{Hom}_\Lambda(X_i,N)\ar[rrr]_\sim^{\eta_{X_0}}\ar[d]_{\mathrm{Hom}_\Lambda(f_{i+1},N)}&&&\mathrm{Hom}_\Gamma(\mathrm{Hom}_\Lambda(X,X_i),\mathrm{Hom}_\Lambda(X,N))\ar[d]_{\mathrm{Hom}_\Gamma(\mathrm{Hom}_\Lambda(X,f_{i+1}),\mathrm{Hom}_\Lambda(X,N))}\\
\mathrm{Hom}_\Lambda(X_{i+1},N)\ar[rrr]^\sim_{\eta_{X_{i+1}}}&&&\mathrm{Hom}_\Gamma(\mathrm{Hom}_\Lambda(X,X_{i+1}),\mathrm{Hom}_\Lambda(X,N)).}$$
Therefore, $\mathrm{Hom}_\Lambda(X_\bullet,N)$ and $\mathrm{Hom}_\Gamma(\mathrm{Hom}_\Lambda(X,X_\bullet),\mathrm{Hom}_\Lambda(X,N))$ are isomorphic as a chain complexes and thus, for any $i\geq 1,$
\begin{equation*}
\begin{split}
\Ext^i_F(M,N) & = \mathrm{H}^i(\Hom_\Lambda(X_\bullet,N))\\
                      & \simeq \mathrm{H}^i(\mathrm{Hom}_\Gamma(\mathrm{Hom}_\Lambda(X,X_\bullet),\mathrm{Hom}_\Lambda(X,N)))\\
                      & = \Ext^i_\Gamma(\mathrm{Hom}_\Lambda(X,M),\mathrm{Hom}_\Lambda(X,N)).
\end{split}
\end{equation*}
\end{proof}

\begin{pro}\label{p1, NAS81, 5.1} Let $X$ be an additive generator in $\modu(\Lambda),$  $f: M\to N$ be in $\modu(\Lambda)$ and 
$\Gamma:=\mathrm{End}_\Lambda(X)^{op}.$ Then, the following statements hold true.
\begin{itemize}
\item[(a)] $f$ is left (right) minimal in $\modu(\Lambda)$ if, and only if, $\mathrm{Hom}_\Lambda(X,f)$ is left (right) minimal in 
$\modu(\Gamma).$

\item[(b)] $f$ is $F$-epic if, and only if, $\mathrm{Hom}_\Lambda(X,f)$ is an epimorphism in $\modu(\Gamma).$
\end{itemize}
\end{pro}

\begin{proof}(a)   Let $f$ be left minimal in $\modu(\Lambda).$ Consider some $g: \mathrm{Hom}_\Lambda(X,N)\to\mathrm{Hom}_\Lambda(X,N)$  in $\modu(\Gamma)$ with 
$g\mathrm{Hom}_\Lambda(X,f)=\mathrm{Hom}_\Lambda(X,f).$ By Lemma \ref{l2, NAS81, 5.1} (a) there exists $k: N\to N$ such that $g=\mathrm{Hom}_\Lambda(X,k)$ and thus $kf=f.$ Therefore $k$ is an isomorphism and then $g$ is an isomorphism.
\

Let  $\mathrm{Hom}_\Lambda(X,f)$ be left minimal in $\modu(\Gamma).$ Consider some $h: N\to N$ with $hf=f.$ Then $\mathrm{Hom}_\Lambda(X,h)\mathrm{Hom}_\Lambda(X,f)=\mathrm{Hom}_\Lambda(X,f)$ and therefore $\mathrm{Hom}_\Lambda(X,h)$ is an isomorphism. Hence, from Lemma \ref{l2, NAS81, 5.1} (a),  it follows that $h$ is an isomorphism.
\

(b)  Let $\mathrm{Hom}_\Lambda(X,f)$ be an epimorphism in $\modu(\Gamma).$ By Lemma \ref{FP-FI} (a,b), there is an  $F$-exact 
sequence $0\to K\to X'\overset{g}{\to}N\to0,$ with $X'\in\mathrm{add}(X).$  Since $\Hom_\Lambda(X,f)$ is 
an epimorphism in $\modu(\Gamma)$ and $\Hom_\Lambda(X,X')\in\proj(\Gamma)$ (see Lemma \ref{ASS06, VI.3.1}) (b)), by Lemma 
\ref{l2, NAS81, 5.1} (a), there exists 
$k: X^m\to M$ such that $g=fk.$ Moreover, by using that $g$ is $F$-epic, we get from Lemma \ref{Sa12, 1.10} (f1) that $f$ is  $F$-epic. 
\end{proof}

We recall that  $\ind(\Lambda)$ denotes the  set of representatives of iso-classes of indecomposables  in $\modu(\Lambda).$

The next result is the relative version of \cite[Proposition 5.1]{AS81}.

\begin{pro}\label{NAS81, 5.1} Let $M\in\modu(\Lambda),$ $X$ be an additive generator in $\modu(\Lambda)$ and $F:=F_{\add(X)}.$ If  
$M$ is $\mathrm{gen}_F$-minimal and $\mathrm{gen}_F(M)$ is an $F$-torsion class, then  $M\in{}^{_F\perp_1}\mathrm{gen}_F(M).$
\end{pro}

\begin{proof} Let $M$ be $\mathrm{gen}_F$-minimal and $\mathrm{gen}_F(M)$ be an $F$-torsion class. Suppose that   $M\notin{}^{_F\perp_1}\mathrm{gen}_F(M)$. Then there are $Y\in\mathrm{ind}(\Lambda)$ such that $Y\;|\;M$, $Z\in\gen_F(M)$ and a non splitting  $F$-exact sequence $\eta: 0\to Z\overset{f}{\to}E\overset{g}{\to}Y\to0.$ Then $g\in\mathrm{rad}_\Lambda(E,Y)$ and thus,  from Lemma \ref{ABM98, 2.5}, it follows that $f$ is  left minimal. Hence, by  Proposition \ref{p1, NAS81, 5.1} (a),  $\mathrm{Hom}_\Lambda(X,f)$ is left minimal. Moreover, from the exact sequence 
$$0\to\mathrm{Hom}_\Lambda(X,Z)\xrightarrow{\mathrm{Hom}_\Lambda(X,f)}\mathrm{Hom}_\Lambda(X,E)\xrightarrow{\mathrm{Hom}_\Lambda(X,g)}\mathrm{Hom}_\Lambda(X,Y)\to0,$$  
it follows  by Lemma \ref{ABM98, 2.5},  that $\mathrm{Hom}_\Lambda(X,g)\in\mathrm{rad_\Gamma}(\mathrm{Hom_\Lambda}(X,E),\mathrm{Hom_\Lambda}(X,Y)),$ where $\Gamma:=\mathrm{End}_\Lambda(X)^{op}.$ Since $\gen_F(M)$ is closed under extensions, from 
 $\eta$ we get that $E\in\gen_F(M).$ Hence, there is an $F$-epic $(h,k): Y^m\oplus(M/Y)^m\to E,$ with $h=(h_1,\cdots,h_m).$ By the fact that $(gh,gk)$ is  $F$-epic (see Lemma \ref{Sa12, 1.10} (f2)), it follows that
$$\small{\xymatrix{(\mathrm{Hom_\Lambda}(X,gh),\mathrm{Hom_\Lambda}(X,gk)): \mathrm{Hom}_\Lambda(X,Y^m)\oplus\mathrm{Hom}_\Lambda(X,(M/Y)^m)\to\mathrm{Hom}_\Lambda(X,Y)}}$$
 is an epimorphism. Therefore
$$\footnotesize{\xymatrix{\mathrm{Hom}_\Lambda(X,Y)=\sum_{i=1}^m\mathrm{Hom}_\Lambda(X,gh_i)(\mathrm{Hom}_\Lambda(X,Y))+\mathrm{Hom}_\Lambda(X,gk)(\mathrm{Hom}_\Lambda(X,(M/Y)^m)).}}$$
Since $\mathrm{Hom}_\Lambda(X,g)\in\mathrm{rad_\Gamma}(\mathrm{Hom_\Lambda}(X,E),\mathrm{Hom_\Lambda}(X,Y)),$ we get from \cite[Lemma A.3.4 (b)]{ASS06} that $\sum_{i=1}^m\mathrm{Hom}_\Lambda(X,gh_i)(\mathrm{Hom}_\Lambda(X,Y))\subseteq
\mathrm{rad}_\Gamma(\Hom_\Lambda (X,Y)).$ Then, by  Nakayama's lemma, it follows that 
$$\mathrm{Hom}_\Lambda(X,Y)=\mathrm{Hom}_\Lambda(X,gk)(\mathrm{Hom}_\Lambda(X,(M/Y)^m));$$
proving that $\mathrm{Hom}_\Lambda(X,gk)$ is an epimorphism in $\modu(\Gamma).$ Then, by Proposition \ref{p1, NAS81, 5.1} (b), we have that $gk: (M/Y)^m\to Y$ is  $F$-epic, which is a contradiction since $M$ is $\mathrm{gen}_F$-minimal. Thus $M\in{}^{_F\perp_1}\mathrm{gen}_F(M)$.
\end{proof}

\begin{lem}\label{l1, NHo83} Let $X$ be  an additive generator in $\modu(\Lambda),$ $F:=F_{\mathrm{add}(X)}$ and $T\in\modu(\Lambda)$ be $F$-tilting. Then, $\gen_F(T)=T^{{}_F\perp_1}.$
\end{lem}
\begin{proof} By Lemma \ref{FP-FI}  (a,b),  $F$ has enough projectives and $\mathcal{P}(F)=\add(X).$ 
\

Let $M\in T^{_F\perp_1}$. Since $F$ has enough projectives, by Lemma \ref{Sa12, 1.10} (b),  there is an $F$-epic $p: X'\to M,$ where $X'\in\add(X).$  Using now that $T$ is $F$-tilting, there is an $F$-exact sequence $0\to X'\to T'\to T''\to 0,$ where 
$T',$ $T''\in\add(T).$ Then, by Lemma \ref{Sa12, 1.10} (d), we get the following $F$-exact and commutative diagram 
$$\xymatrix{0\ar[r]&X'\ar[r]\ar[d]_p&T'\ar[r]\ar[d]^{p'}&T''\ar@{=}[d]\ar[r]&0\\
0\ar[r]&M\ar[r]_i&E\ar[r]&T''\ar[r]&0,}$$
where $p': T'\to E$ is $F$-epic. Since $M\in T^{_F\perp_1},$ the second row in the above diagramm splits and thus there exists 
some $q: E\to M$ such that $qi=1_M.$ Then, by Lemma \ref{Sa12, 1.10} (f1),  $q: E\to M$ is  $F$-epic. Thus, by Lemma \ref{Sa12, 1.10} (f2), $qp': T'\to M$ is  $F$-epic. Then, we conclude that $M\in\mathrm{gen}_F(\add(T))=\gen_F(T)$ 
(see Lemma \ref{Sa12, 1.10} (c)) and thus $T^{_F\perp_1}\subseteq\mathrm{gen}_F(T).$

For $N\in\mathrm{gen}_F(T),$ there is  an $F$-exact sequence $0\to K\to T^n\to N\to0.$ By  applying $\mathrm{Hom}_\Lambda(T,-)$ to the previous exact sequence, we have the exact sequence, $0=\mathrm{Ext^1_F}(T,T^n)\to\mathrm{Ext^1_F}(T,N)\to\mathrm{Ext^2_F}(T,K)=0$ since $T$ is $F$-tilting. Therefore $N\in T^{_F\perp_1}$ and thus $\gen_F(T)\subseteq T^{_F\perp_1}.$
\end{proof}

Now, we are ready to give a proof of {\bf Theorem \ref{NHo83}}:

\begin{proof} (a) $\Rightarrow$ (b) Let $X=\oplus_{i=1}^nX_i$, where each $X_i$ is indecomposable. By Proposition \ref{NAS80, 4.6} (a) there exists $T'\in\mathcal{T}$ such that $\mathcal{T}=\gen_F(T').$ Moreover, from Proposition \ref{p1.1, NAS81, 5.1} there exists $T''\;|\;T'$ which is gen$_F$-minimal and $\mathcal{T}=\gen_F(T'').$ On the other hand, since $\mathcal{T}$ is 
$F$-preenveloping and $\T=\smd(\T)$ (see Lemma \ref{T=addT}), it follows by  Lemma \ref{FP-FI} (c) that 
$\mathcal{I}(F)\subseteq\mathcal{T}.$ Then, by the Proposition \ref{NAS80, 4.6} (b1), for every $k\in[1,n],$ there is an $F$-exact sequence 
$\eta_k: 0\to X_k\overset{i_k}{\to}\overline{T}_k\to\overline{C}_k\to0$ such that $i_k: X_k\to\overline{T}_k$ is a 
$\mathcal{T}$-preenvelope and $\overline{T}_k\in\add(T'')$. Now, by Lemma \ref{Sa12, 1.10} (b), we get the $F$-exact sequence 
$\eta:=\oplus_{k=1}^n\eta_k: 0\to X\overset{i}{\to}\overline{T}\to\overline{C}\to0$, where $\overline{T}:=\oplus_{k=1}^n\overline{T}_k,$ $\overline{C}:=\oplus_{k=1}^n\overline{C}_k$ and $i:=\oplus_{k=1}^ni_k$ is a $\mathcal{T}$-preenvelope. We assert that  
$T:=\overline{T}\oplus\overline{C}$ is $F$-tilting and $\mathcal{T}=\mathrm{gen}_F(T).$ Indeed, we will split the proof of this assertion into  the following parts to verify.

(i) $\mathcal{T}=\mathrm{gen}_F(T)$:

Indeed, to prove (i), it is enough to show that $\mathrm{gen}_F(T)=\mathrm{gen}_F(T'').$ Since $\overline{T}\in\add(T''),$ we have by 
Lemma \ref{Sa12, 1.10} (c) that 
$\mathrm{gen}_F(\overline{T})\subseteq\mathrm{gen}_F(\add(T''))=\mathrm{gen}_F(T'').$ Now, by Lemma \ref{FP-FI} (a,b),   there is an 
$F$-exact sequence $\epsilon: 0\to K\to X^m\overset{g}{\to}T''\to0.$ Furthermore, from the $F$-exact sequence $\eta$ and  Lemma \ref{Sa12, 1.10} (b), we obtain the $F$-exact sequence $\eta^m: 0\to X^m\overset{i^m}{\to}\overline{T}^m\to\overline{C}^m\to0.$ Then,  by Lemma \ref{Sa12, 1.10} (d),  we get the following $F$-exact and commutative diagram in $\modu(\Lambda)$
$$\xymatrix{&&0\ar[d]&0\ar[d]\\
&&K\ar[d]\ar@{=}[r]&K\ar[d]\\
\eta^m:&0\ar[r]&X^m\ar[r]^{i^m}\ar[d]_g&\bar{T}^m\ar[r]\ar[d]^{g'}&\bar{C}^m\ar@{=}[d]\ar[r]&0\\
\lambda:&0\ar[r]&T''\ar[r]_j\ar[d]&E\ar[d]\ar[r]&\bar{C}^m\ar[r]&0\\
&&0&0.}$$

Since $i^m: X^m\to\bar{T}^m$ is a $\mathcal{T}$-preenvelope and $T''\in\mathcal{T}$, we have that $g$ can be factored through $i^m$ and therefore $\lambda$ splits. Thus, there exists $j': E\to T''$ such that $j'j=1_{T''}.$ In particular, by Lemma \ref{Sa12, 1.10} (f1),  $j'$ is $F$-epic. Then, by  Lemma \ref{Sa12, 1.10} (f2),  it follows that $j'g': \overline{T}^m\to T''$ is  $F$-epic and thus 
$\mathrm{gen}_F(T'')\subseteq\mathrm{gen}_F(\overline{T}).$

(ii) $T\in{}^{_F\perp_1}\mathcal{T}$:

Since $T''$ is $\gen_F$-minimal and $\T=\gen_F(T''),$ it follows from Proposition \ref{NAS81, 5.1} that $\overline{T}\in{}^{_F\perp_1}\mathcal{T}$. We show now that $\overline{C}\in{}^{_F\perp_1}\mathcal{T}$. Indeed, by applying $\mathrm{Hom}_\Lambda(-,Y)$, with $Y\in\mathcal{T}$, to the $F$-exact sequence $\eta$ and since $\mathrm{Ext^1_F}(\overline{T},Y)=0,$ we get the exact sequence 
$$\mathrm{Hom}_\Lambda(\overline{T},Y)\xrightarrow{\mathrm{Hom}_\Lambda(i,Y)}\mathrm{Hom}_\Lambda(X,Y)\to\mathrm{Ext^1_F}(\overline{C},Y)\to 0.$$
From the above exact sequence and  since $i$ is a $\mathcal{T}$-preenvelope, it follows that $\mathrm{Ext^1_F}(\overline{C},Y)=0.$

(iii) $T\in{}^{_F\perp_1}T$: 

It follows from (ii), since $T\in\mathcal{T}.$

(iv) For any $Q\in\mathcal{P}(F),$  there is an $F$-exact sequence $0\to Q\to T_0\to T_1\to 0$ with $T_0,T_1\in\add(T).$ 

Indeed, by Lemma \ref{FP-FI} (b), we know that $\mathcal{P}(F)=\add(X).$  Let $Q\in\add(X).$ Since $\{\eta_k\}_{k=1}^n$ is a family of 
$F$-exact sequences and each $X_k$ is indecomposable, by Lemma \ref{Sa12, 1.10} (b), there is an $F$-exact sequence 
$0\to Q\to T_0\to T_1\to0,$ where $T_0\in\mathrm{add}(\overline{T})\subseteq\mathrm{add}(T)$ and $T_1\in\mathrm{add}(\overline{C})\subseteq\mathrm{add}(T)$. 

(v) $\mathrm{pd}_F(T)\leq1$: 

By Lemma \ref{FP-FI} (d),  it is enough to show that $\Ext^2_F(T,-)=0$. Let $M\in\modu(\Lambda)$. By Lemma \ref{FP-FI} (a),  there is an 
$F$-exact sequence $\theta:\;0\to M\to I\to L\to0$, with $I\in\mathcal{I}(F)\subseteq\mathcal{T}.$ Since $\mathcal{T}$ is $F$-torsion, we have that $L\in\mathcal{T}.$ Then, by applying the functor $\mathrm{Hom}_\Lambda(T,-)$ to the $F$-exact sequence $\theta,$ we obtain the following exact sequence $\mathrm{Ext^1_F}(T,L)\to\mathrm{Ext^2_F}(T,M)\to\mathrm{Ext^2_F}(T,I).$ Then, by (ii) we know that $\mathrm{Ext^1_F}(T,L)=0$ and since $\mathrm{Ext^2_F}(T,I)=0,$ we conclude that $\mathrm{Ext^2_F}(T,M)=0$ for any $M\in\modu(\Lambda).$

(b) $\Rightarrow$ (a) Let $T\in\modu(\Lambda)$ be an $F$-tilting such that $\mathcal{T}=\mathrm{gen}_F(T).$ We prove firstly that  $\mathrm{gen}_F(T)$ is an $F$-torsion class. In order to do that, from  Lemma \ref{NAS81, 5.10}, it is enough to show that $\mathcal{T}\subseteq{}^{_F\perp_1}\mathrm{gen}_F(T).$ Let $M\in\mathrm{gen}_F(T)$. Then there is an $F$-exact sequence $0\to K\to T^m\to M\to0,$ and thus, we obtain the exact sequence $0=\mathrm{Ext^1_F}(T,T^m)\to\mathrm{Ext^1_F}(T,M)\to\mathrm{Ext^2_F}(T,K)=0$ since $\mathrm{pd}_F(T)\leq1;$ proving that $\mathrm{gen}_F(T)$ is a $F$-torsion class. 

Now, to show that $\T$ is $F$-preenveloping, by Proposition \ref{NAS80, 4.6} (b), it is enough to prove that $\mathcal{I}(F)\subseteq\mathcal{T}$. Let $I\in\mathcal{I}(F).$ Then, by Lemma \ref{FP-FI} (a,b), there is an $F$-epic $p: X'\to I,$ with $X'\in\add(X)=\mathcal{P}(F).$ Now, since $T$ is $F$-tilting, there is an $F$-exact sequence $\eta: 0\to X'\overset{i}{\to}T'\to T''\to 0,$ with $T',$ $T''\in\add(T).$ Consider the following exact commutative diagram in $\modu(\Lambda)$
$$\xymatrix{\eta:&0\ar[r]&X'\ar[r]^i\ar[d]_p&T'\ar[r]\ar[d]^h&T''\ar@{=}[d]\ar[r]&0\\
\epsilon:&0\ar[r]&I\ar[r]_j&E\ar[r]&T''\ar[r]&0.}$$
Then, by Lemma \ref{Sa12, 1.10} (d),  we get that $h$ is  $F$-epic. On the other hand, by using that $\eta$ is $F$-exact and 
$I\in\mathcal{I}(F),$ we have that $p$ can be factored through $i$, and thus $\epsilon$  splits. Therefore, there exists $j': E\to I$ such that $j'j=1_I.$ Then, by Lemma \ref{Sa12, 1.10} (f1) $j': E\to I$ is $F$-epic. Hence by Lemma \ref{Sa12, 1.10} (f2), we get the 
$F$-epic $j'h:T'\to I.$ Thus  $I\in\mathcal{T}$ since $\mathcal{T}$ is $F$-torsion.
\

Let us show the bijective correspondence between basic  $F$-tilting $\Lambda$-modules and the classes $\mathcal{T}\subseteq\modu(\Lambda)$ which are $F$-torsion and $F$-preenveloping. By the equivalence between (a) and (b) proved above, 
 the assignment $T\mapsto\mathrm{gen}_F(T)$ is surjective. 
\

Consider basic $F$-tilting $\Lambda$-modules   $T$ and $S$  such that $\gen_F(T)=\gen_F(S).$ In particular, there is an $F$-epic 
$p: T^m\to S$ and an $\add(T)$-precover $f: T'\to S$. Hence,  there is some $h: T^m\to T'$ such that $p=fh.$ Since $p$ is  $F$-epic, it follows from Lemma \ref{Sa12, 1.10} (f1) that $f$ is  $F$-epic and thus $\eta: 0\to K\to T'\overset{f}{\to}S\to0$ is an $F$-exact sequence. By  applying  $\mathrm{Hom}_\Lambda(T,-)$ to $\eta,$ and since $\mathrm{Ext^1_F}(T,T')=0,$  we get the exact sequence  
$$\mathrm{Hom}_\Lambda(T,T')\xrightarrow{\mathrm{Hom}_\Lambda(T,f)}\mathrm{Hom}_\Lambda(T,S)\to\mathrm{Ext^1_F}(T,K)\to 0.$$
Therefore  $K\in T^{_F\perp_1}$ since $f$ is an $\add(T)$-precover. Moreover,  by Lemma \ref{l1, NHo83}, we have that $K\in T^{_F\perp_1}=\mathrm{gen}_F(T)=\mathrm{gen}_F(S)=S^{_F\perp_1}$ and thus $\eta$ splits. Therefore $S\in\mathrm{add}(T).$ Similarly, 
it can be shown that $T\in\mathrm{add}(S).$ Then $S\simeq T.$
\end{proof}

\begin{cor}\label{CP1}  Let $\Lambda$ be of finite representation type, $X$ be an additive generator in $\modu(\Lambda),$  $F:=F_{\mathrm{add}(X)}$ and $\mathcal{T}$ be an $F$-torsion class. Then, there is an $F$-tilting $T\in\modu(\Lambda)$ such that 
$\mathcal{T}=\gen_F(T)$ if, and only if, $\tau(X)\oplus D_{\Lambda^{op}}(\Lambda^{op})\in\mathcal{T}.$
\end{cor}

\begin{proof} By Lemma \ref{FP-FI} (b), we have that $\mathcal{I}(F)=\add(\tau(X))\oplus\inj(\Lambda).$ On the other hand, since $\T$ is $F$-torsion, we get by Lemma \ref{T=addT} that $\mathcal{T}=\mathrm{add}(\mathcal{T}).$

$(\Rightarrow)$ It follows from Theorem \ref{NHo83} and Lemma \ref{FP-FI} (c).

$(\Leftarrow)$ By Theorem \ref{NHo83} and Lemma \ref{FP-FI} (c), it is enough to show that $\mathcal{T}$ is preenveloping. Now, since  
$\mathcal{T}=\mathrm{add}(\mathcal{T})$ and $\Lambda$ is of finite representation type, there exists $M\in\modu(\Lambda)$ such that 
$\mathcal{T}=\add(M).$ In particular, $\mathcal{T}$ is preenveloping.
\end{proof}

\section{$F$-silting modules}

Let $\Lambda$ be an Artin $R$-algebra and $X$ be an additive generator in $\modu(\Lambda)$. The main objective of this section is to introduce the $F$-presilting $\Lambda$-modules and to study their relationship with the $\tau$-rigid $\Gamma$-modules, where $\Gamma:=\mathrm{End}_\Lambda(X)^{op}$.

\begin{lem}\label{l,RAIR, 3.4} Let $X$ be an additive generator in $\modu(\Lambda),$ $\Gamma:=\mathrm{End}_\Lambda(X)^{op},$  $F:=F_{\mathrm{add}(X)}$ and $M\in\modu(\Lambda)$. If $\Hom_\Lambda(X,M)$ is a $\tau$-rigid $\Gamma$-module, then 
$\gen_F(M)$ is an $F$-torsion class.
\end{lem}
\begin{proof} By Lemma \ref{NAS81, 5.10},  it is enough to show that $\Ext^1_F(M,\mathrm{gen}_F(M))=0.$ Since $\mathrm{Hom}_\Lambda(X,M)$ is a $\tau$-rigid $\Gamma$-module, we get from \cite[Proposition 5.8]{AS81} that 
$$\Ext^1_\Gamma(\Hom_\Lambda(X,M), \gen(\mathrm{Hom}_\Lambda(X,M))=0.$$ 
Then, by the inclusion $\mathrm{Hom}_\Lambda(X,\mathrm{gen}_F(M))\subseteq\mathrm{gen}(\mathrm{Hom}_\Lambda(X,M))$ and
Lemma \ref{l2, NAS81, 5.1} (b), we get that  $\Ext^1_F(M,\mathrm{gen}_F(M))=0.$
\end{proof}

\begin{lem}\label{RAIR, 3.4} Let $X$ be an additive generator in $\modu(\Lambda),$ $F:=F_{\mathrm{add}(X)}$ and let\\
$\mathrm{P^{-1}_F}(M)\overset{\pi^{-1}_F}{\to}\mathrm{P^0_F}(M)\overset{\pi^0_F}{\to}M\to0$ be the minimal $F$-projective presentation of $M\in\modu(\Lambda).$ Then, for any $N\in\mathrm{mod}(\Lambda),$ the following statements are equivalent.
\begin{itemize}
\item[(a)] $\mathrm{Hom_\Lambda}(\pi^{-1}_F,N): \mathrm{Hom_\Lambda}(\mathrm{P^0_F}(M),N)\to\mathrm{Hom_\Lambda}(\mathrm{P^{-1}_F}(M),N)$ is an epimorphism.
\item[(b)] $\mathrm{P^\bullet_{F\geq-1}}(N)[1]\in\mathrm{P^\bullet_{F\geq-1}}(M)^{\perp_0}$ in $\mathsf{K}^b(\mathrm{add}(X))$.\end{itemize}
\end{lem}
\begin{proof} Let $\mathrm{P^{-1}_F}(N)\overset{\pi'^{-1}_F}{\to}\mathrm{P^0_F}(N)\overset{\pi'^0_F}{\to}N\to0$ be the minimal $F$-projective presentation of $N.$
\

(a) $\Rightarrow$ (b) Let $f^\bullet: \mathrm{P^\bullet_{F\geq-1}}(M)\to\mathrm{P^\bullet_{F\geq-1}}(N)[1]$ be a morphism of chain complexes. Then, $f^\bullet$ comes from a morphism $f: \mathrm{P^{-1}_F}(M)\to\mathrm{P^0_F}(N)$  in $\modu(\Lambda)$ as in the following diagram
$$\xymatrix{\cdots\ar[r]&0\ar[d]\ar[r]&\mathrm{P^{-1}_F}(M)\ar[d]^{f}\ar[r]^{\pi^{-1}_F}&\mathrm{P^0_F}(M)\ar[r]\ar[d]&0\ar[d]\ar[r]&\cdots\\
\cdots\ar[r]&\mathrm{P^{-1}_F}(N)\ar[r]_{\pi'^{-1}_F}&\mathrm{P^0_F}(N)\ar[r]&0\ar[r]&0\ar[r]&\cdots.}$$
By hypothesis, there is $g: \mathrm{P^0_F}(M)\to N$ such that $g\pi^{-1}_F=\pi'^0_Ff.$ Now since $\mathrm{P^0_F}(M)$ is $F$-projective, we have the following commutative diagram
$$\xymatrix{&\mathrm{P^0_F}(M)\ar[d]^{g}\ar@{-->}[dl]_{\exists\; h_0}\\
\mathrm{P^0_F}(N)\ar[r]_{\pi'^0_F}&N.}$$
Let $\mathrm{P^{-1}_F}(N)\overset{p}{\to}\mathrm{Im}(\pi'^{-1}_F)\overset{i}{\to}\mathrm{P^0_F}(N)$ be the factorization of $\pi'^{-1}_F$ through its image. Since $\pi'^0_F(f-h_0\pi^{-1}_F)=\pi'^0_Ff-g\pi^{-1}_F=0,$ there exists $k: \mathrm{P^{-1}_F}(M)\to\mathrm{Im}(\pi'^{-1}_F)$ such that $ik=f-h_0\pi_1.$ Using that $\mathrm{P^{-1}_F}(M)$ is $F$-projective and $p$ is $F$-epic, we have the following commutative diagram
$$\xymatrix{&\mathrm{P^{-1}_F}(M)\ar[d]^{k}\ar@{-->}[dl]_{\exists\; h_1}\\
\mathrm{P^{-1}_F}(N)\ar[r]_{p}&\mathrm{Im}(\pi'^{-1}_F).}$$
Then $\pi'^{-1}_Fh_1=ik=f-h_0\pi^{-1}_F$ and thus $f^\bullet\sim0;$ proving  (b).

(b) $\Rightarrow$ (a) Let $f: \mathrm{P^{-1}_F}(M)\to N.$ Since $\mathrm{P^{-1}_F}(M)$ is $F$-projective and $\pi'^0_F$ is $F$-epic, we get the following commutative diagram
$$\xymatrix{&\mathrm{P^{-1}_F}(M)\ar@{-->}[dl]_{\exists\;g}\ar[d]^{f}\\
\mathrm{P^0_F}(N)\ar[r]_{\pi'^0_F}&N.}$$
By hypothesis, there are $h_0: \mathrm{P^0_F}(M)\to\mathrm{P^0_F}(N)$ and $h_1: \mathrm{P^{-1}_F}(M)\to\mathrm{P^{-1}_F}(N)$ such that $\pi'^{-1}_Fh_1+h_0\pi^{-1}_F=g.$ Then  $f=\pi'^0_Fg=\pi'^0_F(\pi'^{-1}_Fh_1+h_0\pi^{-1}_F)=\pi'^0_Fh_0\pi^{-1}_F.$
\end{proof}

Now, we give the proof of {\bf Proposition \ref{RAIR, 3.5}}:
\begin{proof}
Let  $\cdots\to\mathrm{P^{-2}_F}(M)\xrightarrow{\pi^{-2}_F}\mathrm{P^{-1}_F}(M) \xrightarrow{\pi^{-1}_F}\mathrm{P^0_F}(M)\xrightarrow{\pi^0_F}M\to0$ be the minimal $F$-projective resolution  of $M\in\modu(\Lambda).$  From this resolution,  we get the chain complex $(\mathrm{P^{\bullet}_F}(M),M)$ which is as follows
$$(*)\quad 0\to (\mathrm{P^0_F}(M),M)\xrightarrow{(\pi^{-1}_F,M)}(\mathrm{P^{-1}_F}(M),M)\xrightarrow{(\pi^{-2}_F,M)}(\mathrm{P^{-2}_F}(M),M)\to\cdots.$$
(a) $\Rightarrow$ (b) By Lemma \ref{RAIR, 3.4}, we know that $(\pi^{-1}_F,M)$ is an epimorphism and thus 
$\Ext^1_F(M,M)=H^1((\mathrm{P^{\bullet}_F}(M),M))=0.$ Moreover, for any $\theta\in (\mathrm{P^{-1}_F}(M),M)$ there is some $\theta'\in (\mathrm{P^0_F}(M),M)$ such that $\theta=\theta'\pi^{-1}_F.$ Therefore $(\pi^{-2}_F,M)(\theta)=\theta\pi^{-2}_F=\theta'\pi^{-1}_F\pi^{-2}_F=0.$
\

(b) $\Rightarrow$ (a)  From the complex $(*),$ we get that $\Ima\, (\pi^{-1}_F,M)=\Ker\,(\pi^{-2}_F,M)=(\mathrm{P^{-1}_F}(M),M)$ since $\Hom_\Lambda(\pi^{-2}_F, M)=0$ and $\Ext^1_F(M,M)=0.$ Thus, by Lemma \ref{RAIR, 3.4}, we get (a).
\

(b) $\Rightarrow$ (c) In particular, by the proved above, we get that $M$ is $F$-presilting. Let us show that $M\in{}^{_F\perp_1}\gen_F(M).$ Consider $N\in\mathrm{gen}_F(M).$ Hence, there is an $F$-epic $p: M^n\to N.$  
In order to show that $\mathrm{Ext^1_F}(M,N)=0,$ by using the $F$-projective resolution of $M,$  it is enough to prove that 
$\mathrm{Hom}_\Lambda(\pi^{-1}_F,N): \mathrm{Hom}_\Lambda(\mathrm{P^0_F}(M),N)\to\mathrm{Hom}_\Lambda(\mathrm{P^{-1}_F}(M),N)$ is an epimorphism. Let $f: \mathrm{P^{-1}_F}(M)\to N$. Now, since $p: M^n\to N$ is  $F$-epic, there exists $g: \mathrm{P^{-1}_F}(M)\to M^n$ such that $pg=f.$ Then, by Lemma \ref{RAIR, 3.4} there exists $h: \mathrm{P^0_F}(M)\to M^n$ such that $g=h\pi^{-1}_F.$ Hence $\mathrm{Hom}_\Lambda(\pi^{-1}_F,N)$ is an epimorphism since $(ph)\pi^{-1}_F=pg=f.$
\

(a) $\Leftrightarrow$ (d) Let  $\mathrm{P^{-1}_F}(M)\to\mathrm{P^0_F}(M)\to M\to0$ be the minimal $F$-projective presentation of $M.$ Then, by Lemma \ref{FP-FI} (b), Proposition \ref{ASS06, VI.3.1} (b) and Proposition \ref{p1, NAS81, 5.1} (a), it follows that  
{\small $\mathrm{Hom}_\Lambda(X,\mathrm{P^{-1}_F}(M))\to\mathrm{Hom}_\Lambda(X,\mathrm{P^0_F}(M))\to\mathrm{Hom}_\Lambda(X,M)\to0$}
 is a minimal projective presentation in $\mathrm{mod}(\Gamma).$ Moreover, from Proposition \ref{ASS06, VI.3.1}, it can be shown  that 
 $\mathrm{P^\bullet_{F\geq-1}}(M)$ is presilting in $\mathsf{K}^b(\mathrm{add}(X))$ if, and only if, $\mathrm{P^\bullet_{\geq-1}}(\mathrm{Hom}_\Lambda(X,M))$ is presilting in $\mathsf{K}^b(\mathrm{proj}(\Gamma)).$ Finaly,  by \cite[Lemma 3.4]{AIR14}, the latter is equivalent to $\mathrm{Hom}_\Lambda(X,M)$ being $\tau$-rigid.
\end{proof}

In what follows, we give an example of an Artin algebra $\Lambda$ and an additive generator $X\in\modu(\Lambda)$ such that $\Lambda$ is not $F$-admissible, for 
$F:=F_{\mathrm{add}(X)}.$

\begin{ex}\label{Ej-noFadm} Let $k$ be a field, $\xymatrix{Q:=\circ^1\ar[r]^\alpha&\circ^2\ar@(ul,ur)^{\beta}}$ and $\Lambda:=kQ/\langle\beta^2\rangle.$ The Auslander-Reiten quiver of $\mathrm{mod}(\Lambda)$ is the following one
$$
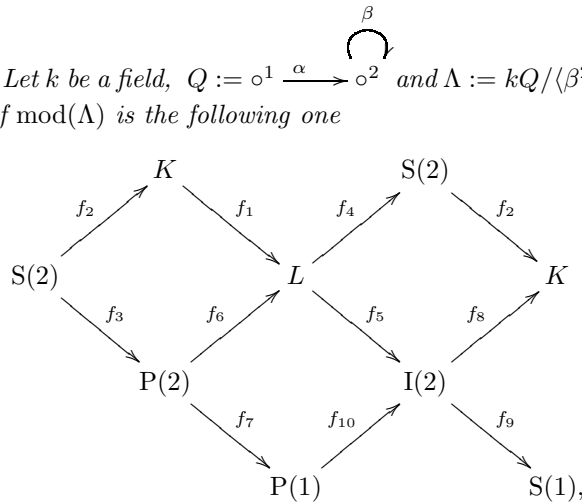
\xymatrix{& K\ar[rd]^{f_1}&&\mathrm{S(2)}\ar[rd]^{f_2}\\
\mathrm{S(2)}\ar[ru]^{f_2}\ar[rd]^{f_3}&& L\ar[ru]^{f_4}\ar[rd]^{f_5}&&K\\
&\mathrm{P(2)}\ar[ru]^{f_6}\ar[rd]^{f_7}&&\mathrm{I(2)}\ar[ru]^{f_8}\ar[rd]^{f_9}\\
&&\mathrm{P(1)}\ar[ru]^{f_{10}}&&\mathrm{S(1),}}$$
where $K:=P(1)/S(2)$ and $L$ is the unique indecomposable with dimension vector $(1,2)$ and $L/\mathrm{rad}(L)=\mathrm{S(1)\oplus S(2)}.$ 

Consider $X:=\Lambda\oplus K,$  $F:=F_{\mathrm{add}(X)}$ and  $M:=I(2).$ Note that $\Ext^1_F(M,M)=0$ since $M$ is injective. 
We assert that $M$ is not $F$-presilting.  Indeed, the 
$(\mathcal{P}(F),F)$-cover of $\mathrm{I(2)}$ is given by the  $F$-exact sequence 
$$\eta: 0\to\mathrm{S(2)}\xrightarrow{\begin{pmatrix} f_7f_3\\ -f_2\end{pmatrix}}\mathrm{P(1)}\oplus K\xrightarrow{\begin{pmatrix}f_{10}   & f_5f_1\end{pmatrix}}\mathrm{I(2)}\to 0,$$
and the $(\mathcal{P}(F),F)$-cover of $\mathrm{S(2)}$ is given by the $F$-exact sequence 
$$0\to\mathrm{S(2)}\xrightarrow{f_3}\mathrm{P(2)}\xrightarrow{f_4f_6}\mathrm{S(2)}\to0.$$ Then,  the minimal $F$-projective presentation of $\mathrm{I(2)}$ is 
$$\mathrm{P(2)}\xrightarrow{\pi^{-1}_F}\mathrm{P(1)}\oplus K\xrightarrow{\pi^0_F}\mathrm{I(2)}\to0,$$
where $\pi^0_F=\begin{pmatrix}f_{10} & f_5f_1\end{pmatrix}$ and $\pi^{-1}_F=\begin{pmatrix} f_7f_3f_4f_6   \\  -f_2f_4f_6\end{pmatrix}.$

Suppose that $M=\mathrm{I(2)}$ is $F$-presilting.  Consider the irreducible morphism 
$f_7: \mathrm{P(2)}\to\mathrm{P(1)}$ and the map 
$f^\bullet: \mathrm{P^\bullet_{F\geq-1}}(\mathrm{I(2)})\to\mathrm{P^\bullet_{F\geq-1}}(\mathrm{I(2)})[1]$ given by the following  morphism of chain complexes
$$\xymatrix{0\ar[r]\ar[d]_0&\mathrm{P(2)}\ar[r]^{\pi^{-1}_F}\ar[d]^{\begin{pmatrix}f_7\\0\end{pmatrix}}&\mathrm{P(1)}\oplus K\ar[d]_0\\
\mathrm{P(2)}\ar[r]_{\pi^{-1}_F}&\mathrm{P(1)}\oplus K\ar[r]&0.}$$
Since  $\Hom_{K^b(\add(X))}(\mathrm{P^{\bullet}_{F\geq -1}}(\mathrm{I(2)}), \mathrm{P^{\bullet}_{F\geq -1}}(\mathrm{I(2)})[1])=0,$   
there are morphisms $s': \mathrm{P(1)}\oplus K\to\mathrm{P(1)}\oplus K$ and $s: \mathrm{P(2)}\to\mathrm{P(2)}$ such that $\begin{pmatrix}f_7\\0\end{pmatrix}= s'\pi^{-1}_F+\pi^{-1}_Fs.$ Moreover, by using that\\  $\pi^{-1}_F=\begin{pmatrix}f_7f_3f_4f_6\\-f_2f_4f_6\end{pmatrix}\in\mathrm{rad}^2_\Lambda(\mathrm{P(2)},\mathrm{P(1)}\oplus K)$ and $\mathrm{rad}^2_\Lambda$ is an ideal, we obtain that 
$\begin{pmatrix}f_7\\0\end{pmatrix}=s'\pi^{-1}_F+\pi^{-1}_Fs\in\mathrm{rad}^2_\Lambda(\mathrm{P(2)},\mathrm{P(1)}\oplus K)$ and thus  
$f_7\in\mathrm{rad}^2_\Lambda(\mathrm{P(2)},\mathrm{P(1)}),$ which is a contradiction with the fact that $f_7$ is irreducible. 
Therefore $\mathrm{I(2)}$ is not $F$-presilting. Finally, it can be shown  that 
$\gen_F(\mathrm{I(2)}))=\add (\mathrm{I(2)}\oplus\mathrm{S(1)}).$ 
\end{ex}

\begin{lem}\label{l1, RAIR, 3.6}  $X[0]$ is a silting complex in $\mathsf{K}^b(\mathrm{add}(X)),$ for any $X\in\modu(\Lambda).$
\end{lem}

\begin{proof} Let $X\in\modu(\Lambda).$ It is clear that $X[0]$ is presilting in $\mathsf{K}^b(\mathrm{add}(X)).$ Then, it is enough to show 
 that $Z^\bullet\in\mathrm{thick}(X[0])$ $\forall$ $Z^\bullet\in\mathsf{K}^b(\mathrm{add}(X)).$ We do the proof of that by induction on the 
 length $\ell(Z^\bullet)$ of the complex $Z^\bullet\in\mathrm{thick}(X[0])$. It can be assumed that $Z^\bullet\neq0$ and it is concentrated in degree $0.$ If $\ell(Z^\bullet)=0$, then there is nothing to prove. 
\

Let $n:=\ell(Z^\bullet)\geq1.$ Then $Z^\bullet$ has the following shape 
$$\cdots\to 0\to Z^{-(n+1)}\xrightarrow{d^{-(n+1)}}Z^{-n}\to\cdots\to Z^0\to0\to\cdots.$$
Consider the map $f^\bullet: Z^{-(n+1)}[n]\to Z'^\bullet$ in $\mathsf{K}^b(\mathrm{add}(X))$ which is induced by $d^{-n-1}$ as in the following diagram
$$\xymatrix{\cdots\ar[r]&0\ar[d]\ar[r]& Z^{-(n+1)}\ar[r]\ar[d]_{d^{-(n+1)}}&0\ar[r]\ar[d]&\cdots\ar[r]& 0\ar[d]\ar[r]&0\cdots\\
\cdots\ar[r]&0\ar[r]& Z^{-n}\ar[r]& Z^{-(n-1)}\ar[r]&\cdots\ar[r]& Z^0\ar[r]&0\cdots.}$$
Since $Z^{-(n+1)}[n]\in\mathrm{thick}(X[0]),$ by the inductive hypothesis we get  $Z'^\bullet\in\mathrm{thick}(X[0]).$ Finally,  we 
conclude that $Z^\bullet\in\mathrm{thick}(X[0])$  since $Z^\bullet=C(f^\bullet).$ 
\end{proof}

 In what follows, we use the following notation. For the case of a cochain complex of $\Lambda$ modules which is concentrate in degrees $0$ and $-1,$ take for example $M^\bullet:=\cdots\to0\to M^{-1}\to M^0\to0\to\cdots,$  for simplicity we write $(M^{-1}\to M^0)$ to denote 
the entire complex $M^\bullet.$

\begin{lem} Let $X$ be an additive generator in $\modu(\Lambda),$  $F:=F_{\mathrm{add}(X)}$, $M\in\mathrm{ind}(\Lambda)$ and $\mathrm{P^{-1}_F}(M)\xrightarrow{\pi^{-1}_F}\mathrm{P^0_F}(M)\xrightarrow{\pi^0_F}M\to0$ be the minimal $F$-projective presentation of $M.$ Then $(\mathrm{P^{-1}_F}(M)\xrightarrow{\pi^{-1}_F}\mathrm{P^0_F}(M))\in\mathrm{ind}(\mathsf{K}^b(\mathrm{add}(X))$. In particular, if $N\in\mathrm{mod}(\Lambda)$ then $\mathrm{rk}(\mathrm{P^{-1}_F}(N)\xrightarrow{\pi^{-1}_F}\mathrm{P^0_F}(N))=\mathrm{rk}(N).$
\end{lem}

\begin{proof} Let $(\mathrm{P^{-1}_F}(M)\xrightarrow{\pi^{-1}_F}\mathrm{P^0_F}(M))=(P^{-1}\xrightarrow{d^{-1}}P^0)\oplus(P'^{-1}\xrightarrow{d'^{-1}}P'^0).$ Consider $d^{-1}=i\bar{d}^{-1}$ and $d'^{-1}=j\bar{d'}^{-1}$ which are, respectively, the factorizations through its image of $d^{-1}$ and $d'^{-1}.$ Note that $M=\mathrm{Coker}(d^{-1})\oplus\mathrm{Coker}(d'^{-1})$ and since $M\in\mathrm{ind}(\Lambda)$, without loss of generality it can be assumed that $\mathrm{Coker}(d'^{-1})=0.$
\

 Let us show that $P^{-1}\xrightarrow{d^{-1}}P^0\xrightarrow{p}M\to0$ is an $F$-projective presentation of $M.$ By Lemma \ref{FP-FI} (a,b) and Lemma \ref{Sa12, 1.10} (e),  it is enough to prove that $\mathrm{Hom}(P,\bar{d}^{-1})$ and $\mathrm{Hom}(P,p)$ are epimorphisms 
 $\forall$ $P\in\add(X).$ Let $P\in\add(X)$, $f: P\to M$ and $g: P\to\mathrm{Im}(d^{-1})$ in $\modu(\Lambda).$ Since 
 {\small $$\pi^{-1}_F=\begin{pmatrix}d^{-1}&0\\0&d'^{-1}\end{pmatrix}:P^{-1}\oplus P'^{-1}\to P^0\oplus P'^0\,\text{ and }\;
 \pi^0_F=\begin{pmatrix}p&0\end{pmatrix} :P^0\oplus P'^0\to M,$$}  there are $f': P\to P^0$, $f'': P\to P'^0$, $g': P\to P^{-1}$ and $g'': P\to P'^{-1}$ such that $\begin{pmatrix}p&0\end{pmatrix}\begin{pmatrix}f'\\f''\end{pmatrix}=f$ and $\begin{pmatrix}\bar{d}^{-1}&0\\0&\bar{d'}^{-1}\end{pmatrix}\begin{pmatrix}g'\\g''\end{pmatrix}=\begin{pmatrix}g\\0\end{pmatrix}$. Hence $P^{-1}\overset{d^{-1}}{\to}P^0\overset{p}{\to}M\to0$ is an $F$-projective presentation of $M$ since $f=pf'$ and $g=\bar{d}^{-1}g'.$ Furthermore, we have the following $F$-exact  and commutative diagram 
$$\xymatrix{\mathrm{P^{-1}_F}(M)\ar_{l}[d]\ar^{\pi^{-1}_F}[rr]&&\mathrm{P^0_F}(M)\ar^{k}[d]\ar^{\pi^0_F}[rr]&&M\ar@{=}[d]\ar[rr]&&0\\
P^{-1}\ar_{l'}[d]\ar^{d^{-1}}[rr]&&P^0\ar^{k'}[d]\ar^{p}[rr]&&M\ar@{=}[d]\ar[rr]&&0\\
\mathrm{P^{-1}_F}(M)\ar_{\pi^{-1}_F}[rr]&&\mathrm{P^0_F}(M)\ar_{\pi^0_F}[rr]&&M\ar[rr]&&0.}$$
From the right minimality of $\pi^0_F$ and $\pi^{-1}_F,$ it follows that $k$ and $l$ are split-monomorphisms. Moreover, since $P^{-1}\;|\;\mathrm{P^{-1}_F}(M)$ and $P^0\;|\;\mathrm{P^0_F}(M),$ we conclude that $k$ and $l$ are isomorphisms. Hence $P^{-1}\oplus P'^{-1}=\mathrm{P^{-1}_F}(M)=P^{-1}$ and $P^0\oplus P'^0=\mathrm{P^0_F}(M)=P^0,$ and thus $(P'^{-1}\overset{d'^{-1}}{\to}P'^0)=(0\to0)$.\end{proof}

\begin{defi} Let $X$  be an additive generator in $\modu(\Lambda)$ and $F:=F_{\mathrm{add}(X)}.$ For a pair  $(M,X'),$  with $M\in\modu(\Lambda)$ and $X'\in\add(X),$ we say that:
\begin{itemize}
\item[$\mathrm{(a)}$]  $(M,X')$ is $\mathbf{F}$-$\mathbf{presilting}$ if $M$ is $F$-presilting and $\mathrm{Hom}_\Lambda(X',M)=0;$
\item[$\mathrm{(b)}$]  $(M,X')$ is $\mathbf{support}$ $F$-$\mathbf{silting}$ if $(M,X')$ is $F$-presilting and 
$$\mathrm{rk}(X)=\mathrm{rk}(M)+\mathrm{rk}(X').$$
\end{itemize}
\end{defi}

\begin{cor}\label{c1, RAIR, 3.6} For an additive generator $X$  in $\modu(\Lambda),$ $F:=F_{\mathrm{add}(X)}$ and $\Gamma:=\mathrm{End}_\Lambda(X)^{op},$ the following statements hold true.
\begin{itemize}
\item[$\mathrm{(a)}$]  Let $(M,X')$ be an $F$-presilting $(support$ $F$-silting$)$ pair in $\modu(\Lambda).$ Then the pair 
{\small $(\mathrm{Hom}_\Lambda(X,M),\mathrm{Hom}_\Lambda(X,X'))$} is $\tau$-rigid $(support$ $\tau$-$tilting)$ in $\modu(\Gamma).$
\item[$\mathrm{(b)}$] The map $(M,X')\mapsto(\mathrm{Hom}_\Lambda(X,M),\mathrm{Hom}_\Lambda(X,X')),$ between the set of $F$-presilting pairs in $\modu(\Lambda)$ and the $\tau$-rigid pairs in $\modu(\Gamma),$ induces an injective map between their respective iso-classes.
\end{itemize}
\end{cor}

\begin{proof} Let $(M,X')$ be an $F$-presilting pair in $\modu(\Lambda)$ and let
$$\Phi(M,X'):=(\mathrm{Hom}_\Lambda(X,M),\mathrm{Hom}_\Lambda(X,X')).$$
Then, by Proposition \ref{RAIR, 3.5} and Lemma \ref{l2, NAS81, 5.1} (a), it follows that the pair $\Phi(M,X')$ is $\tau$-rigid  in 
$\modu(\Gamma).$ Now, assume that $(M,X')$ is support $F$-tilting. Then, by Lemma \ref{l2, NAS81, 5.1} (a) and  Proposition \ref{ASS06, VI.3.1} (b),  we get the equalities
$$\mathrm{rk}(\mathrm{Hom}_\Lambda(X,M))+\mathrm{rk}(\mathrm{Hom}_\Lambda(X,X'))=\mathrm{rk}(M)+\mathrm{rk}(X')=\mathrm{rk}(X)=\mathrm{rk}(\Gamma).$$ Therefore the pair $\Phi(M,X')$ is support $\tau$-tilting in $\modu(\Gamma);$ proving (a).

In order to prove (b), let $(M',X'')$ be an $F$-presilting pair in $\modu(\Lambda).$ Note that if $(M,X')\cong(M',X'')$ then 
$\Phi(M,X')\cong\Phi(M',X'')$. Hence $\Phi$ induces a well-defined map between the respective iso-classes. We assert that this map is 
injective. Indeed, suppose that $\Phi(M,X')\cong\Phi(M',X'')$.  In particular
$$(*)\quad\mathrm{Hom}_\Lambda(X,M)\cong\mathrm{Hom}_\Lambda(X,M')\;\text{ and }\;\mathrm{Hom}_\Lambda(X,X')\cong\mathrm{Hom}_\Lambda(X,X'').$$
On the other hand, from Lemma \ref{l2, NAS81, 5.1} (a), we have that 
$$\begin{matrix}(**)&\mathrm{Hom}_\Lambda(M,M')\cong\mathrm{Hom}_\Gamma(\mathrm{Hom}_\Lambda(X,M),\mathrm{Hom}_\Lambda(X,M'))&and\\
&\mathrm{Hom}_\Lambda(X',X'')\cong\mathrm{Hom}_\Gamma(\mathrm{Hom}_\Lambda(X,X'),\mathrm{Hom}_\Lambda(X,X''))
\end{matrix}$$
Therefore, by applying $(*)$ to the pair of isomorphisms in $(**),$ we conclude that $M\cong M'$ and $X'\cong X''$ and thus 
$(M,X')\cong(M',X'').$
\end{proof}

\begin{defi} Let $0\neq M\in \modu(\Lambda)$ and $M=\oplus_{i=1}^n M^{m_i}_i$ be its decomposition into pairwise non isomorphic indecomposable 
$\Lambda$-modules $M_1,M_2,\cdots, M_n.$  The basic part of $M$ is the $\Lambda$-module $b(M):=\oplus_{i=1}^n M_i.$
\end{defi}

\begin{lem}\label{basicM} For an additive generator $X$  in $\modu(\Lambda)$ and  $F:=F_{\mathrm{add}(X)},$ the following statements hold true.
\begin{itemize}
\item[$\mathrm{(a)}$] If $M\in\modu(\Lambda)$ is $\gen_F$-minimal, then $b(M)$ is $\gen_F$-minimal.
\item[$\mathrm{(b)}$] $\gen_F(M)=\gen_F(b(M))$ $\forall\,M\in\modu(\Lambda).$
\item[$\mathrm{(c)}$] Let $M,N\in\modu(\Lambda)$ and  $\Gamma:=\End_\Lambda(X)^{op}.$ If $\gen_F(M)=\gen_F(N),$  then $\gen(\Hom_\Lambda(X,M))=\gen(\Hom_\Lambda(X,N))$ in $\modu(\Gamma).$
\end{itemize}
\end{lem}
\begin{proof} (a) Assume that $M$ is $\gen_F$-minimal.  Let $N\mid b(M)$ be such that\\ $N\in\gen_F(b(M)/N).$ We must show that $N=0.$  Indeed, since $b(M)/N$ is a 
direct summand of $M/N,$ we get from Lemma \ref{Sa12, 1.10} (b) that $N\in\gen_F(b(M)/N)\subseteq\gen_F(M/N),$  and thus 
$N=0$ since $M$ is $\gen_F$-minimal.
\

(b) It follows from Lemma \ref{Sa12, 1.10} (c) since $\add(b(M))=\add(M).$
\

(c) Let $H\in \gen(\Hom_\Lambda(X,N)).$ Then, there is an epimorphism $$p: \Hom_\Lambda(X,N^n)\to H$$ in $\modu(\Gamma).$ Since $\mathrm{gen}_F(M)=\mathrm{gen}_F(N),$ there is an $F$-epic $$q: M^m\to N^n.$$ Thus, by Proposition \ref{p1, NAS81, 5.1},  it follows that 
$$\mathrm{Hom}_\Lambda(X,q): \mathrm{Hom}_\Lambda(X,M^m)\to\mathrm{Hom}_\Lambda(X,N^n)$$ is an epimorphism in $\modu(\Gamma)$ and so 
$p\mathrm{Hom}_\Lambda(X,q): \mathrm{Hom}_\Lambda(X,M)^m\to H$ is also
 an epimorphism, proving that  $\mathrm{gen}(\mathrm{Hom}_\Lambda(X,N))\subseteq\mathrm{gen}(\mathrm{Hom}_\Lambda(X,M)).$ In a similar way, it can be shown that  $\mathrm{gen}(\mathrm{Hom}_\Lambda(X,M))\subseteq\mathrm{gen}(\mathrm{Hom}_\Lambda(X,N)).$
\end{proof}

\begin{pro} \label{AuxP}  Let $X$  be an additive generator in $\modu(\Lambda),$  $F:=F_{\mathrm{add}(X)}$ and $\T\subseteq\modu(\Lambda)$ be a non zero, preenveloping and $F$-torsion class. Then, there is some $M\in\modu(\Lambda)$ which is basic, 
$\T=\gen_F(M)$ and $M\in{}^{{}_F\perp_1}\T.$
\end{pro}
\begin{proof} By Proposition \ref{NAS80, 4.6} (a), there is some $T\in\modu(\Lambda)$ such that $\T=\gen_F(T).$ Note that $T\neq 0$ since $\T\neq 0.$ Then, by Proposition \ref{p1.1, NAS81, 5.1} and Lemma \ref{basicM}, there is some $M\in\modu(\Lambda)$ which is basic,  
$\gen_F$-minimal and $\gen_F(M)=\gen_F(T)=\T.$ Finally, from Proposition \ref{NAS81, 5.1}, we get that $M\in{}^{{}_F\perp_1}\T.$
\end{proof}

The following is an example of an $F$-presilting $\Lambda$-module $M$ such that $\mathrm{gen}_F(M)$ is not preenveloping. Note that, for a 
$\tau$-rigid $\Lambda$-module $N,$ it is known that $\mathrm{gen}(N)$ is a torsion class which is functorially finite \cite[Theorem 5.10]{AS81}.  In the example below, we use the description of the Kronecker algebra category $\modu(\Lambda)$ given in \cite[Section VIII.7]{ARS95}. 

\begin{ex}\label{Kronecker} Let $k$ be a field, $\xymatrix{Q:=\circ^1 \ar@<1ex>[r]^\alpha \ar@<-1ex>[r]_\beta& \circ^2}$ and $\Lambda=kQ.$ Consider the following representations on $Q:$ $\xymatrix{J_n:=k^{n+1}\ar@<1ex>[rrr]^{(1_{k^n},0)}\ar@<-1ex>[rrr]_{(0,1_{k^n})}&&&k^n}$ $n\in\mathbb{N}^+;$
$$\xymatrix{R_{p,n}:=k^n\ar@<1ex>[rrr]^{1_{k^n}}\ar@<-1ex>[rrr]_{J_{n,p}}&&&k^n\quad n\in\mathbb{N}^+,\;p\in\mathbb{P}^1(k);}$$
$X:=\Lambda\oplus R_{(1,1),1}$ and $F:=F_{\mathrm{add}(X)}.$ For each $n\in\mathbb{N},$ we have the following exact sequence 
$0\to\mathrm{S(2)}\to R_{(1,1),1}^{n+1}\overset{g_n}{\to}J_n\to0$  in $\modu(\Lambda).$ Since $g_n$ is an $\mathrm{add}(R_{(1,1),1})$-precover, this sequence is $F$-exact. Thus $$\mathcal{Y}:=\mathrm{gen}_F(R_{(1,1),1})=\widetilde{\mathrm{inj}}(\Lambda)\oplus\mathrm{add}(R_{(1,1),1}),$$
where $\widetilde{\mathrm{inj}}(\Lambda)$ is the pre-injective component of $\modu(\Lambda).$
\

We assert that $R_{(1,0),1}$ does not have an $\mathcal{Y}$-preenvelope. Suppose there exists an $\mathcal{Y}$-preenvelope $f: R_{(1,0),1}\to Y$. Since $\mathrm{Hom}_\Lambda(R_{(1,0),1},R_{(1,1),1})=0,$ it follows that $Y\in\widetilde{\mathrm{inj}}(\Lambda).$ Thus,  there exists $r\in\mathbb{N}^+$ such that $\mathrm{Hom}_\Lambda(Y, J_r)=0.$ Consider the following exact sequence $0\to\mathrm{S(2)}\to R_{(1,0),1}^{r+1}\overset{p}{\to}J_r\to0$ in $\modu(\Lambda).$ Since $\oplus_{i=1}^{r+1}f: R_{(1,0),1}^{r+1}\to Y^{r+1}$ is an $\mathcal{Y}$-preenvelope, there exists $g: Y^{r+1}\to J_r$ such that $p=g(\oplus_{i=1}^{r+1}f)=0,$ which is a contradiction. Therefore $\mathcal{Y}$ is not a preenveloping class.
\end{ex}

Now, we are ready to give a proof of {\bf Theorem \ref{teo-princip2}}:

\begin{proof}  Let $M\in\modu(\Lambda)$ be a special $F$-presilting  such that $\gen_F(M)$ is preenveloping. Then, by Proposition \ref{RAIR, 3.5} and Lemma 
\ref{l,RAIR, 3.4}, we get that $\gen_F(M)$ is $F$-torsion and thus the map 
$$\mathrm{F}\text{-}\mathrm{presilt}\text{-}\mathrm{esp}(\Lambda)\to \mathrm{F}\text{-}\mathrm{tor}\text{-}\mathrm{preen}(\Lambda),\; [M]\mapsto\mathrm{gen}_F(M)$$
is well defined. Let us show that this map is injective. Indeed, let $M$ and $N$ be special $F$-presilting $\Lambda$-modules such that 
$\mathrm{gen}_F(M)=\mathrm{gen}_F(N).$ Then, by Lemma \ref{basicM} (c), we have that $\mathbb{P}(\mathrm{gen}(\mathrm{Hom}_\Lambda(X,M)))=\mathbb{P}(\mathrm{gen}(\mathrm{Hom}_\Lambda(X,N))).$ Hence by \cite[Theorem 5.10]{AS81}, it follows that 
$\Hom_\Lambda(X,M)\;|\;\mathbb{P}(\mathrm{gen}(\mathrm{Hom}_\Lambda(X,N))).$ Let $M'|M$ be indecomposable. Then $\Hom_\Lambda(X,N\oplus M')\;|\;\mathbb{P}(\mathrm{gen}(\mathrm{Hom}_\Lambda(X,N))),$ and since $N$ is special $F$-presilting, it follows that $M'\in\mathrm{add}(N).$ Thus 
$\add(M)\subseteq \add(N).$ Analogously, it can be shown that $\add(N)\subseteq \add(M)$ and then $M\simeq N$ since $M$ and $N$ are basic.
\

Let us show now that the above correspondence is surjective. Consider $\mathcal{T}\in\mathrm{F}$-$\mathrm{tor}$-$\mathrm{preen}(\Lambda)$ and $\Gamma:=\End_\Lambda(X)^{op}.$ Then, by Proposition \ref{AuxP} and Proposition \ref{RAIR, 3.5}, there is some basic $F$-presilting $T\in\T$ such that $\gen_F(T)=\T$ and  $\Hom_\Lambda(X,T)$ is $\tau$-rigid in 
$\modu(\Gamma).$  Hence, from \cite[Theorem 5.10]{AS81},  it follows that $\mathrm{gen}(\mathrm{Hom}_\Lambda(X,T))$ is a functorially finite torsion class in 
$\modu(\Gamma)$ and 
$$(i):\quad \Hom_\Lambda(X,T)\mid \mathbb{P}(\gen(\mathrm{Hom}_\Lambda(X,T))).$$
Furthermore, by  \cite[Theorem 2.7]{AIR14}, we get that 
\begin{center}
$(ii):\quad \mathbb{P}(\gen(\mathrm{Hom}_\Lambda(X,T)))$ is $\tau$-rigid  in $\modu(\Gamma);$ and   
\end{center}
\begin{center}
$(iii):\quad \mathrm{gen}(\mathbb{P}(\mathrm{gen}(\mathrm{Hom}_\Lambda(X,T))))=\mathrm{gen}(\mathrm{Hom}_\Lambda(X,T)).$
\end{center}
By $(i),$ there exists  some $N\in\modu(\Lambda)$ which is zero or basic and satisfies the following two conditions:
\begin{itemize}
\item[(C1)] $\add(T)\cap\add(N)=0$ and 
$\mathrm{Hom}_\Lambda(X,T\oplus N)\;|\;\mathbb{P}(\gen(\mathrm{Hom}_\Lambda(X,T))),$
\item[(C2)] $\forall$ $K\in\modu(\Lambda)-\{0\}$  with $\add(K)\cap\add(T\oplus N)=0$, we have that 
 \begin{center}
 $\mathrm{Hom}_\Lambda(X,T\oplus N\oplus K)\notin\mathrm{add}(\mathbb{P}(\gen(\mathrm{Hom}_\Lambda(X,T))).$
 \end{center}
\end{itemize}
 
\noindent Consider $M:=T\oplus N.$ Note that $M$ is basic. Moreover, by (C1) and $(ii),$ we get that $\Hom_\Lambda(X,M)$ is $\tau$-rigid in $\modu(\Gamma).$ Hence, by Proposition \ref{RAIR, 3.5}, $M$ is $F$-presilting. 
 \
 
 Let us show that $\mathcal{T}=\gen_F(M).$

$(\subseteq)$ Since $M=T\oplus N$ and $\mathrm{gen}_F(T)=\mathcal{T},$ it follows that $\mathcal{T}\subseteq\mathrm{gen}_F(M).$

$(\supseteq)$ Let $p: M^n\to L$ be an $F$-epic. By Proposition \ref{p1, NAS81, 5.1} (b),  $\mathrm{Hom}_\Lambda(X,p):\Hom_\Lambda(X,M)^n\to\Hom_\Lambda(X,L)$ is an epimorphism in 
$\modu(\Gamma)$ and then, by (C1) and $(iii),$ we get 
$$\Hom_\Lambda(X,L)\in\gen(\mathbb{P}(\mathrm{gen}(\mathrm{Hom}_\Lambda(X,T))))=\mathrm{gen}(\mathrm{Hom}_\Lambda(X,T)).$$
In particular, there is an epimorphism $\Hom_\Lambda(X,T^m)\to \Hom_\Lambda(X,L).$ Therefore, by Lemma \ref{l2, NAS81, 5.1} (a) and Proposition \ref{p1, NAS81, 5.1} (b), there is an $F$-epic $q: T^m\to L;$ proving that $L\in\mathrm{gen}_F(T)=\mathcal{T}.$

Once we have that $\gen_F(M)=\T=\gen_F(T),$ by Lemma \ref{basicM} (c), we obtain that $\gen(\Hom_\Lambda(X,M))=\gen(\Hom_\Lambda(X,T)).$ Then, from (C2), we conclude that $M$ is $F$-special; proving the result.
\end{proof}

In what follows, we give an example where we can apply Theorem \ref{teo-princip2} but not the Theorem \ref{NHo83}.

\begin{ex}\label{Ejem4} Let $k$ a field, $\xymatrix{Q:=\circ^1\ar[r]^\alpha&\circ^2\ar[r]^\beta&\circ^3}$, $\Lambda=kQ/R^2$ and\\ $X:=\Lambda\oplus\mathrm{S(2)}.$ Then,  the Auslander-Reiten quiver of $\modu(\Lambda)$ is as follows
$$
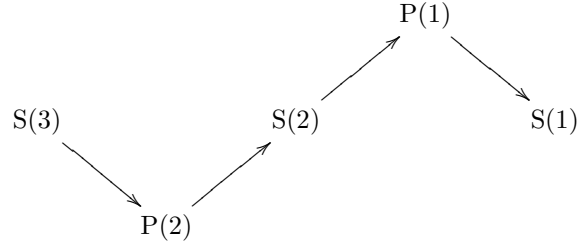
\xymatrix{&&&\mathrm{P(1)}\ar[dr]\\
\mathrm{S(3)}\ar[dr]&&\mathrm{S(2)}\ar[ru]&&\mathrm{S(1)}\\
&\mathrm{P(2)}\ar[ur]}$$
Consider the class $\T:=\gen(P(1)).$ Let us show the following:
\begin{enumerate}
\item $\gldim_F(\Lambda)=1.$ In particular, $\Lambda$ is $F$-admissible.\\
{\rm Inded, since $\pd_F(X)=0$ and $\modu(\Lambda)=\add(X\oplus S(1)),$ it is enough to show that $\pd_F(S(1))=1.$ But, this follows from the Auslander-Reiten quiver since $0\to S(2)\to P(1)\to S(1)\to 0$ is an $F$-exact sequence.}

\item $\T$ is $F$-torsion and preenveloping but not $F$-preenveloping. In particular, from Theorem \ref{NHo83}, $\T\neq \gen_F(T)$ for any $F$-tilting $T\in\modu(\Lambda).$ However, by Theorem \ref{teo-princip2}, there is some $M\in\mathrm{F}\text{-}\mathrm{presilt}\text{-}\mathrm{esp}(\Lambda)$ with $\T=\gen_F(M).$\\
{\rm Indeed, note firstly that $\T=\add(P(1)\oplus S(1))=\gen_F(P(1))$ and thus $\T$ is functorially finite. Moreover, since   $P(1)$ is projective, by Lemma \ref{NAS81, 5.10}, we get that $\T$ is $F$-torsion. On the other hand, by Lemma \ref{FP-FI} (b), 
$\mathcal{I}(F)=\add(S(3))\oplus\inj(\Lambda)\not\subseteq \gen_F(P(1)).$ Thus, by Lemma \ref{FP-FI} (c), we get that $\T=\gen_F(P(1))$ is not $F$-preenveloping.}

\item $\T\neq \gen_F(M)$ for any support $F$-silting pair $(M, X').$\\
{\rm   Indeed, the only basic $F$-presilting pairs $(M,X')$ such that $\mathrm{gen}_F(M)=\mathrm{gen}_F(\mathrm{P(1)})$ are the following:
\begin{center}
$(\mathrm{P(1)},\mathrm{S(3)})$, $(\mathrm{P(1)\oplus S(1)},\mathrm{S(3)})$, $(\mathrm{P(1)},0)$ and $(\mathrm{P(1)\oplus S(1)},0).$
\end{center}
 Note that none of the above pairs is support $F$-silting
since $\mathrm{rk}(M)+\mathrm{rk}(X')\leq 3$ and $\mathrm{rk}(X)=4.$}
\end{enumerate}
\end{ex}

\footnotesize

\vskip3mm \noindent Luis Mart\'inez Gonz\'alez\\ Instituto de Matem\'aticas\\ Universidad Nacional Aut\'onoma de M\'exico.\\ 
Circuito Exterior, Ciudad Universitaria\\
C.P. 04510, M\'exico, D.F. MEXICO.\\ {\tt luis.martinez@matem.unam.mx}

\vskip3mm \noindent Octavio Mendoza Hern\'andez\\ Instituto de Matem\'aticas\\ Universidad Nacional Aut\'onoma de M\'exico.\\ 
Circuito Exterior, Ciudad Universitaria\\
C.P. 04510, M\'exico, D.F. MEXICO.\\ {\tt omendoza@matem.unam.mx}
\end{document}